\documentclass[12pt]{amsart}
\usepackage{amssymb}
\usepackage{amsbsy}
\usepackage{amscd}

\usepackage{mathrsfs}
\usepackage{eucal}
\usepackage[enableskew]{youngtab}

\usepackage[all]{xy}
\usepackage[dvipdfm]{graphicx}
\usepackage{pstricks}

\makeatletter
\renewcommand{\thesubsection}{\thesection(\@roman\c@subsection)}
\makeatother
\usepackage{verbatim}
\usepackage{version}
\usepackage{color}
\newenvironment{NB}{
\color{red}{\bf NB}. \footnotesize
}{}
\excludeversion{NB}

\newtheorem{Theorem}[equation]{Theorem}
\newtheorem{Corollary}[equation]{Corollary}
\newtheorem{Lemma}[equation]{Lemma}
\newtheorem{Proposition}[equation]{Proposition}

\theoremstyle{definition}
\newtheorem{Definition}[equation]{Definition}

\theoremstyle{remark}
\newtheorem{Remark}[equation]{Remark}

\numberwithin{equation}{section}

\newcommand{\thmref}[1]{Theorem~\ref{#1}}
\newcommand{\secref}[1]{\S\ref{#1}}
\newcommand{\lemref}[1]{Lemma~\ref{#1}}
\newcommand{\propref}[1]{Proposition~\ref{#1}}

\newcommand{\subsecref}[1]{\S\ref{#1}}

\newcommand{\remref}[1]{Remark~\ref{#1}}

\newcommand{\lsp}[2]{{\mskip-.3mu}{}^{#1}\mskip-1mu{#2}}
\newcommand{\defeq}{\overset{\operatorname{\scriptstyle def.}}{=}}
\newcommand{\C}{{\mathbb C}}
\newcommand{\Z}{{\mathbb Z}}
\newcommand{\R}{{\mathbb R}}
\newcommand{\proj}{{\mathbb P}}

\newcommand{\SU}{\operatorname{\rm SU}}
\newcommand{\GL}{\operatorname{GL}}

\newcommand{\algsl}{\operatorname{\mathfrak{sl}}} 

\newcommand{\gl}{\operatorname{\mathfrak{gl}}}

\newcommand{\End}{\operatorname{End}}
\newcommand{\Hom}{\operatorname{Hom}}
\newcommand{\Ext}{\operatorname{Ext}}
\newcommand{\Ker}{\operatorname{Ker}}

\newcommand{\Ima}{\operatorname{Im}}

\newcommand{\rank}{\operatorname{rank}}

\newcommand{\tr}{\operatorname{tr}}

\newcommand{\id}{\operatorname{id}}
\newcommand{\ve}{\varepsilon}
\newcommand{\gr}{\operatorname{gr}}
\newcommand{\dslash}{/\!\!/} 
\newcommand{\bv}{{\mathbf v}} 
\newcommand{\bw}{{\mathbf w}} 
\newcommand{\bM}{{\mathbf M}} 
\newcommand{\bL}{{\mathbf L}}
\newcommand{\bE}{{\mathbf E}}
\newcommand{\Q}{\mathscr Q}
\newcommand{\sO}{\mathscr O}
\newcommand{\reg}{{\operatorname{reg}}}
\newcommand{\shfO}{\mathcal O}
\newcommand{\hT}{\mathbb T}
\newcommand{\tI}{\widetilde I}
\newcommand{\ch}{\operatorname{ch}}
\newcommand{\tG}{\mathbb G}
\newcommand{\Pa}{\mathscr P}

\renewcommand{\MR}[1]{}

\usepackage[bookmarks=false]{hyperref}

\begin{document}
\author{Hiraku Nakajima}
\title[Handsaw quiver varieties and finite $W$-algebras]
{Handsaw quiver varieties and finite $W$-algebras
}
\address{Research Institute for Mathematical Sciences,
Kyoto University, Kyoto 606-8502,
Japan}
\email{nakajima@kurims.kyoto-u.ac.jp}
\thanks{Supported by the Grant-in-aid
for Scientific Research (No.23340005), JSPS, Japan.
}
\begin{abstract}
  Following Braverman-Finkelberg-Feigin-Rybnikov (arXiv:1008.3655), we
  study the convolution algebra of a handsaw quiver variety, a.k.a.\ a
  parabolic Laumon space, and a finite $W$-algebra of type $A$. This
  is a finite analog of the AGT conjecture on $4$-dimensional
  supersymmetric Yang-Mills theory with surface operators. Our new
  observation is that the $\C^*$-fixed point set of a handsaw quiver
  variety is isomorphic to a graded quiver variety of type $A$, which
  was introduced by the author in connection with the representation
  theory of a quantum affine algebra. As an application, simple
  modules of the $W$-algebra are described in terms of $IC$ sheaves of
  graded quiver varieties of type $A$, which were known to be related
  to Kazhdan-Lusztig polynomials. This gives a new proof of a
  conjecture by Brundan-Kleshchev on composition multiplicities on
  Verma modules, which was proved by Losev, in a wider context, by a
  different method.
\end{abstract}

\maketitle
\tableofcontents

\section*{Introduction}

Alday, Gaiotto and Tachikawa proposed a conjecture relating an $N=2$
$4$-dimensional supersymmetric Yang-Mills theory and a $2$-dimensional
conformal field theory \cite{AGT}. This AGT conjecture predicts an
existence of an action of a Virasoro algebra on the direct sum of the
equivariant cohomology group of framed moduli spaces of
$SU(2)$-instantons on $\R^4$ with various instanton numbers.

Rigorously speaking, we need to use the Uhlenbeck partial
compactification and its equivariant intersection cohomology group,
but we will ignore this point in the introduction.

There are lots of variants of this conjecture.
When $SU(2)$ is replaced by a compact simple Lie group $K$, the
Virasoro algebra is replaced by the $W$-algebra associated with the
Langlands dual of the corresponding affine Lie algebra
$(((\operatorname{Lie}K)^\C)^{\mathrm{aff}})^L$.
If instantons have singularities along $x$-axis (with surface
operators in physics terminology), the $W$-algebra is replaced by its
variant associated with a nilpotent orbit in
$((\operatorname{Lie}K)^\C)^L$\footnote{This is only a rough idea, and
  a precise formulation of the conjecture is not understood yet even
  in physics literature. The author thanks Yuji Tachikawa for an
  explanation of this problem.}, describing the type of singularities.
A proof of this conjecture for $K = SU(n)$ without surface operators
will be given by Maulik-Okounkov \cite{MO}.

In a finite analog of the AGT conjecture with surface operators, one
considers spaces of based maps from $\proj^1$ to a partial flag variety
$G/P$ in the geometric side, while in the representation theory side
the finite $W$-algebra $W(\mathfrak g,e)$ associated with the
principal nilpotent element $e$ in the Lie algebra of the Levi
subgroup of $P$ appears.
This conjecture is proved recently by Braverman, Feigin, Finkelberg
and Rybnikov for $G = GL(N)$ with arbitrary $P$ \cite{BFFR}, using
Brundan-Kleshchev's presentation of the finite $W$-algebra via
generators and relations \cite{BruKle-alg}.

The AGT conjecture is close to the author's study of quiver varieties
and their relation to representation theory of Kac-Moody Lie algebras
\cite{Na-quiver,Na-alg}
and their quantum loop algebras \cite{Na-qaff}. In fact, quiver
varieties (of an affine type) are framed moduli spaces of
$\Gamma$-equivariant $U(N)$-instantons on $\R^4$, where $\Gamma$ is a
finite subgroup of $\SU(2)$. In this theory, $\Gamma$ cannot be the
trivial group, as their is no corresponding affine Kac-Moody Lie
algebra under the McKay correspondence. Hence it cannot be applied to
the setting of the AGT conjecture\footnote{%
  One exception is that it is known that the Heisenberg algebra acts
  on the cohomology of $U(1)$-instanton moduli spaces on $\R^4$ (=
  Hilbert schemes of points in $\C^2$) (see \cite{Lecture}). It
  suggests that the $W$-algebra for $\gl_1$ should be the Heisenberg
  algebra.}.

The space of based maps has a quiver description in type $A$, hence is
called a {\it handsaw quiver variety\/} by Finkelberg-Rybnikov
\cite{FR}. A trivial, but crucial new observation in this paper, is
that the $\C^*$-fixed point set of a handsaw quiver variety is
isomorphic to a graded quiver variety of type $A$, which is also the
$\C^*$-fixed point set of a quiver variety.
This observation is useful as we can apply various results in graded
quiver varieties to handsaw quiver varieties.

As an example of such applications, we give a formula of the
composition multiplicity of a simple module in a Verma module of the
finite $W$-algebra of type $A$ in terms of intersection cohomology
groups of graded quiver varieties.
On the other hand, the same intersection cohomology groups control the
composition multiplicity of representations of quantum affine
algebras of type $A$ \cite{Na-qaff} and Yangians \cite{Varagnolo}.
Then this composition multiplicity is equal to a certain
Kazhdan-Lusztig polynomial evaluated at $1$ thanks to Arakawa's result
\cite{Arakawa}. This gives a new proof of a conjecture by
Brundan-Kleshchev \cite[Conj.~7.17]{BruKle}, which was proved earlier
by Losev \cite[Th.~4.1, Th.~4.3]{Losev}, in a wider context, by a
different method.

Other applications, such as a handsaw version of a geometric
realization of tensor products \cite{Na-tensor}, which gives the Miura
transform of a finite $W$-algebra, generalizations to chainsaw quiver
varieties, will be discussed in a subsequent publication.

The paper is organized as follows: in \secref{sec:Walg} we recall
Brundan-Kleshchev's results on the representation theory of finite
$W$-algebras. In \secref{sec:handsaw} we introduce handsaw quiver
varieties and explain their basic properties. In \secref{sec:Laumon}
we give a self-contained proof of the fact that handsaw quiver
varieties are isomorphic to the parabolic Laumon spaces. In
\secref{sec:fixed} we study torus fixed points in handsaw quiver
varieties and compute Betti numbers as an application. In
\secref{sec:Hecke} we introduce a subvariety in the product of two
handsaw quiver varieties, which we call a Hecke correspondence. It
will give generators of the finite $W$-algebra.
In \secref{sec:conv} we introduce a Steinberg type variety, and define
its convolution algebra. We modify the construction in \cite{BFFR} to
define a homomorphism from the finite $W$-algebra to the convolution
algebra.
In \secref{sec:std} we identify standard modules of the convolution
algebra with Verma modules of the finite $W$-algebra. In the course of
the proof, we give a geometric interpretation of the Gelfand-Tsetlin
character of standard modules in the same spirit as a geometric
interpretation of $q$--characters via graded quiver varieties.
In \secref{sec:simple} we give the composition multiplicity formula in
terms of intersection cohomology sheaves of graded quiver varieties.

\subsection*{Notation}

For a linear map $\alpha\colon V\to W$, its transpose $W^*\to V^*$ is
denoted by $\lsp{t}\alpha$.

Let $X$ be a variety endowed with an action of a connected reductive
group $G$. The equivariant Borel-Moore homology group of $X$ is
denoted by
\begin{equation*}
   H^G_*(X),
\end{equation*}
where the coefficient field is $\C$. It is a module over the
equivariant cohomology $H^*_G(\mathrm{pt})$ of a single point, which
is isomorphic to the algebra of $G$-invariant polynomial functions on
the Lie algebra $\mathfrak g$ of $G$. We denote $H^*_G(\mathrm{pt})$
by $S(G)$, and its field of fractions by $\mathcal S(G)$.

The character of $\C^*$ given by the identity map is denoted by
$q$. We usually omit the symbol $\otimes$ for a tensor product of $q$
and a representation or an equivariant vector bundle.

\section{Shifted Yangians and finite $W$-algebras}\label{sec:Walg}

We fix notation on shifted Yangian and finite $W$-algebras following
\cite{BruKle} in this section.
We slightly change terminologies following \cite{Na-qaff} so that
analogy with quantum affine algebras will be apparent.

\subsection{Pyramid}
We fix a {\it pyramid $\pi$ of level $l$\/}, that is, a sequence of
positive integers $q_1,\dots,q_l$ such that
\begin{equation*}
   q_1 \le\cdots\le q_k,\quad q_{k+1} \ge \cdots \ge q_l
\end{equation*}
for some fixed integer $0\le k\le l$. We set $n = \max (q_1,\dots
q_l)$. We define a sequence $0\le p_1\le p_2\le \cdots \le p_n = l$ so
that $p_{i}$ is the number of entries $q_c$ with $q_c \ge
n-i+1$.

We draw the pyramid as a diagram consisting of $q_i$ boxes in the
$i^{\mathrm{th}}$ column. See Figure~\ref{fig:pyramid}. Then $p_i$ is
the $i^{\mathrm{th}}$ {\it row length\/}, that is, the number of boxes
in the $i^{\mathrm{th}}$ row.

\begin{figure}[htbp]
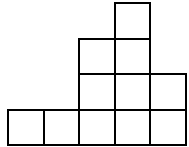

  \centering
  \begin{equation*}
    \young(:::\ :,::\ \ :,::\ \ \ ,\ \ \ \ \ )
  \end{equation*}
  \caption{Pyramid for $\pi=(1,1,3,4,2)$}
  \label{fig:pyramid}
\end{figure}

Let $N$ be the total number of boxes, i.e., $N = \sum p_i = \sum q_j$.

For $0 < i < n$ we set
\begin{equation*}
  \begin{split}
  s_{i+1,i} & \defeq \# \{ c = 1,\dots, k \mid n - q_c = i\},
\\
  s_{i,i+1} & \defeq \# \{ c = k+1, \dots, l \mid n - q_c = i\}.
  \end{split}
\end{equation*}
\begin{NB}
  Therefore we have
  \begin{equation*}
    s_{i+1,i} + s_{i,i+1} = p_{i+1} - p_i.
  \end{equation*}
  Note that $p_{i+1} - p_i$ is the difference of numbers of boxes in
  the $(i+1)^{\mathrm{th}}$ row and $i^{\mathrm{th}}$ row. We divide
  the pyramid into two sides, the left side is the $1^{\mathrm{st}}$,
  \dots, $k^{\mathrm{th}}$ rows and the right side is
  $(k+1)^{\mathrm{th}}$, \dots, $n^{\mathrm{th}}$ rows. Then
  $s_{i+1,i}$ is the number of boxes in the left side, and
  $s_{i,i+1}$ is in the right side.
\end{NB}%
We then define the {\it shift matrix\/} $\sigma = (s_{ij})_{1\le i,j\le
  n}$ associated with $\pi$ so that $s_{ii} = 0$ and
\begin{equation*}
  s_{ij} + s_{jk} = s_{ik}
\end{equation*}
whenever $|i-j| + |j-k| = |i-k|$.
\begin{NB}
We fix a positive integer $n$ and a {\it shift matrix\/} $\sigma =
(s_{ij})_{1\le i,j\le n}$, that is a matrix of nonnegative integers
such that
\begin{equation*}
  s_{ij} + s_{jk} = s_{ik}
\end{equation*}
whenever $|i-j| + |j-k| = |i-k|$.
\end{NB}

Note that the pyramid is recovered from the shift matrix $\sigma$ and
the level $l = p_n$.

\subsection{Generators and relations}\label{subsec:rel}

Suppose that a shift matrix $\sigma = (s_{ij})$ is given.
Let $I = \{ 1, 2, \dots, n-1\}$, $\tI = I\sqcup \{n\}$.
The {\it shifted Yangian\/} $Y_n(\sigma)$ is a $\C$-graded algebra
with generators $D_i^{(r)}$ ($i\in \tI$, $r > 0$), $E_i^{(r)}$ ($i\in
I$, $r > s_{i,i+1}$), $F_i^{(r)}$ ($i\in I$, $r > s_{i+1,i}$),
subjected to the following relations:
{\allowdisplaybreaks
\begin{gather}
  [D_i^{(r)}, D_j^{(s)}] = 0,
\\
  [E_i^{(r)}, F_j^{(s)}] = \delta_{ij} \sum_{t=0}^{r+s-1}
  \widetilde D_i^{(t)} D_{i+1}^{(r+s-1-t)},
\\
  [D_i^{(r)}, E_j^{(s)}] = (\delta_{ij} - \delta_{i,j+1})
  \sum_{t=0}^{r-1} D_i^{(t)} E_j^{(r+s-1-t)},
\\
  [D_i^{(r)}, F_j^{(s)}] = (\delta_{i,j+1}-\delta_{ij})
  \sum_{t=0}^{r-1} F_j^{(r+s-1-t)} D_i^{(t)},
\\
  [E_i^{(r)}, E_i^{(s+1)}] - [E_i^{(r+1)},E_i^{(s)}] = 
   E_i^{(r)} E_i^{(s)} + E_i^{(s)} E_i^{(r)},
\\
  [F_i^{(r+1)},F_i^{(s)}] - [F_i^{(r)}, F_i^{(s+1)}] = 
   F_i^{(r)} F_i^{(s)} + F_i^{(s)} F_i^{(r)},
\\
  [E_i^{(r)}, E_{i+1}^{(s+1)}] - [E_i^{(r+1)}, E_{i+1}^{(s)}] = 
  - E_i^{(r)} E_{i+1}^{(s)},
\\
  [F_i^{(r+1)}, F_{i+1}^{(s)}] - [F_i^{(r)}, F_{i+1}^{(s+1)}] = 
  - F_{i+1}^{(s)}F_i^{(r)},
\\
  [E_i^{(r)}, E_j^{(s)}] = 0 \quad\text{if $|i-j| > 1$},
\\
  [F_i^{(r)}, F_j^{(s)}] = 0 \quad\text{if $|i-j| > 1$},
\\
  [E_i^{(r)}, [E_{i}^{(s)}, E_j^{(t)}]] 
  + [E_i^{(s)}, [E_i^{(r)}, E_j^{(t)}]] = 0 \quad\text{if $|i-j|=1$},
\\
  [F_i^{(r)}, [F_{i}^{(s)}, F_j^{(t)}]] 
  + [F_i^{(s)}, [F_i^{(r)}, F_j^{(t)}]] = 0 \quad\text{if $|i-j|=1$},
\end{gather}
where} $D_i^{(0)} = 1$ and $\widetilde D_i^{(r)}$ is determined
recursively by the relation
\(
   \sum_{t=0}^r D_i^{(t)} \widetilde D_i^{(r-t)} = \delta_{r0}.
\)

It is known that two shifted Yangians $Y_n(\sigma)$, $Y_n(\dot\sigma)$
are isomorphic under an explicit homomorphism \cite[\S2.3]{BruKle} when
$\sigma$ and $\dot\sigma$ satisfy $s_{i,i+1} + s_{i+1,i} = \dot
s_{i,i+1} + \dot s_{i+1,i}$. This condition is satisfied when $\sigma$
and $\dot\sigma$ come from pyramids $\pi$, $\dot\pi$ which have the
same row lengths.
\begin{NB}
  \begin{equation*}
    \iota (E_i^{(r)}) = (-1)^{s_{i,i+1}-\dot s_{i,i+1}}
    E_i^{(r-s_{i,i+1}+\dot s_{i,i+1})},\qquad
    \iota (F_i^{(r)}) = (-1)^{s_{i+1,i}-\dot s_{i+1,i}}
    E_i^{(r-s_{i+1,i}+\dot s_{i+1,i})}.
  \end{equation*}
\end{NB}

\subsection{Rees algebra}

Let us define {\it higher root\/} elements inductively by
\begin{alignat*}{2}
    E_{i,i+1}^{(r)} & \defeq E_i^{(r)}, & \quad
    E_{i,j}^{(r)} & \defeq [E_{i,j-1}^{(r-s_{j-1,j})},E_{j-1}^{(s_{j-1,j}+1)}],
\\
    F_{i,i+1}^{(r)} & \defeq F_i^{(r)}, & \quad
    F_{i,j}^{(r)} & \defeq [F_{j-1}^{(s_{j,j-1}+1)},F_{i,j-1}^{(r-s_{j,j-1})}].
\end{alignat*}

We introduce a filtration $F_0 Y_n(\sigma) \subset F_1
Y_n(\sigma)\subset \cdots$ on $Y_n(\sigma)$ so that $F_d Y_n(\sigma)$
is the span of all monomials in $F_{i,j}^{(r)}$, $F_{i,j}^{(r)}$,
$D_i^{(r)}$ of degree $\le d$, where we assign degree $r$ for all
$F_{i,j}^{(r)}$, $F_{i,j}^{(r)}$, $D_i^{(r)}$.
Then $Y_n(\sigma)$ is a filtered algebra.

Let $Y^\hbar_n(\sigma)$ be the associated Rees algebra, i.e.,
\begin{equation*}
  Y^\hbar_n(\sigma)\defeq \bigoplus_{d\ge 0} \hbar^d F_d Y_n(\sigma).
\end{equation*}
It is a graded algebra over $\C[\hbar]$.
\begin{NB}
For $a\in F_d Y_n(\sigma)$, we define
\(
  \hbar (\hbar^d a) = \hbar^{d+1} a
  \in \hbar^{d+1} F_{d+1} Y_n(\sigma).
\)
\end{NB}%
We denote $\hbar^r E_i^{(r)}\in \hbar^r F_r Y_n(\sigma)$, considered
as an element in $Y_n^{\hbar}(\sigma)$, by $E_i^{(r)}$.

The higher root vector in $Y_n^\hbar(\sigma)$ is given by
\begin{equation}\label{eq:root}
  \hbar E_{i,j}^{(r)} = [E_{i,j-1}^{(r-s_{j-1,j})},E_{j-1}^{(s_{j-1,j}+1)}]
\end{equation}
in $Y^\hbar_n(\sigma)$.

\begin{NB}
  In \cite{BFFR}, $Y^\hbar_n(\sigma)$ is defined by generators and
  relations. But we must be careful whether
  \begin{equation*}
    E_{i,i+2}^{(r)} = \frac1\hbar [E_i^{(r-s_{i+1,i+2})},E_{i+1}^{(s_{i+1,i+2}+1)}]
  \end{equation*}
is contained in a convolution algebra or not.
\end{NB}

The specialization $Y^\hbar_n(\sigma)/\hbar Y^\hbar_n(\sigma) =
Y^\hbar_n(\sigma)\otimes_{\C[\hbar]} \C$ at $\hbar = 0$ is the graded
algebra $\operatorname{gr} Y_n(\sigma) = \bigoplus_{d \ge 0} F_d
Y_n(\sigma)/F_{d-1} Y_n(\sigma)$, where we set $F_{-1} Y_n(\sigma) =
0$.
It is known that $\gr Y_n(\sigma)$ is the free commutative algebra on
generators $\gr_r E_{i,j}^{(r)}$ ($s_{i,j} < r$), $\gr_r F_{i,j}^{(r)}$
($s_{j,i} < r$), $\gr_r D_i^{(r)}$ ($0 < r$) (\cite[Th.~5.2]{BruKle-alg}).

\begin{Remark}
  In \cite{BFFR} the Rees algebra with respect to the {\it shifted\/}
  Kazhdan filtration $F'_d Y_n(\sigma) \defeq F_{d+1} Y_n(\sigma)$ is
  considered. However it seems to the author that one of the defining
  relations \cite[(3.2)]{BFFR} is not compatible with this grading.
  \begin{NB}
    It is true that they get the Rees algebra for the shifted Kazhdan
    filtration, but the defining relation was wrong.
  \end{NB}
\end{Remark}

\subsection{Finite $W$-algebra}

By \cite{BruKle-alg} the finite $W$-algebra $W(\gl_N,e)$ associated with a
nilpotent matrix $e$ whose Jordan type is given by the partition $(p_n,\dots,p_1)$ is isomorphic to the quotient of
$Y_n(\sigma)$ divided by the two sided ideal generated by 
\(
   \{ D_1^{(r)} \mid r > p_1 \}.
\)
Following \cite{BruKle} we denote this quotient by $W(\pi)$, where $\pi$
is the pyramid corresponding to $\sigma$ and $p_1$ (or $p_n$).

The isomorphism $Y_n(\sigma)\cong Y_n(\dot\sigma)$ mentioned above
descends to quotients, so the $W$-algebra depends only on the
conjugacy class of $e$ up to an isomorphism.

The filtration $F_\bullet Y_n(\sigma)$ induces a filtration on
$W(\pi)$, which is the same as the Kazhdan filtration
\cite{BruKle-alg}. We denote by $W^\hbar(\pi)$ the corresponding Rees
algebra.

It is known that $W^\hbar(\pi)/\hbar W^\hbar(\pi)$ is isomorphic to
the polynomial algebra in variables $\gr_r E_{i,j}^{(r)}$ ($s_{i,j}<r
\le S_{ij}$), $\gr_r F_{i,j}^{(r)}$ ($s_{j,i} < r \le S_{j,i}$),
$\gr_r D_i^{(r)}$ ($s_{i,i} < r \le S_{i,i}$) where $S_{i,j} = s_{i,j}
+ p_{\min(i,j)}$ (\cite[Th.~6.2]{BruKle-alg}). It is further isomorphic to
the coordinate ring of the Slodowy slice to the nilpotent orbit
through $e$.

When $e=0$ (i.e., $p_1=\cdots=p_n=1$), it is known that $W(\gl_N,e)$ is the universal enveloping algebra $U(\gl_N)$ of $\gl_N$. On the other hand, if $e$ is regular nilpotent (i.e., $p_1=\cdots=p_{n-1} =0$, $p_n = N$), $W(\gl_N,e)$ is the center of $U(\gl_N)$.

\subsection{Center of $W(\pi)$}

Let
\begin{equation}\label{eq:genD}
  \begin{split}
  D_i(u) & \defeq \sum_{r=0}^\infty D_i^{(r)} u^{-r},
\\
  C_i(u) & \equiv \sum_{r=0}^\infty C_i^{(r)} u^{-r} \defeq 
  D_1(u) D_2(u-1) \cdots D_i(u-i+1),
\\
  Z_N(u) & \defeq u^{p_1} (u-1)^{p_2} \cdots (u - n+1)^{p_n} C_n(u).
  \end{split}
\end{equation}
We follow the notation in \cite{BruKle}, and the above are denoted by
${\mathsf d}_i(u)$, ${\mathsf a}_i(u)$, ${\mathsf A}_n(u)$
respectively in \cite{BFFR}.
\begin{NB}
\begin{gather*}
  {\mathsf d}_i(u) \defeq \sum_{r\ge 0} d_i^{(r)} u^{-r}, \quad
  {\mathsf a}_i(u) \defeq {\mathsf d}_1(u) {\mathsf d}_2(u-\hbar)
  \cdots {\mathsf d}_i(u-(i-1)\hbar),
\\
  {\mathsf A}_i(u) \defeq u^{p_1} (u-\hbar)^{p_2}\cdots
  (u-(i-1)\hbar)^{p_i} {\mathsf a}_i(u).
\end{gather*}
\end{NB}%

It was shown that the elements $C_n^{(1)}$, $C_n^{(2)}$, \dots\ are
algebraically independent and generate the center $Z(Y_n(\sigma))$ of
the shifted Yangian $Y_n(\sigma)$ \cite[Th.~2.6]{BruKle}.

It was also shown that $Z_N(u)$ is a polynomial of degree $N = p_1 +
\cdots + p_n$ in $W(\pi)$, and the center $Z(W(\pi))$ of $W(\pi)$ is a
polynomial algebra in generators $Z_N^{(1)}$, \dots, $Z_N^{(N)}$, the
coefficients of $Z_N(u)$ \cite[Th.~6.10]{BruKle}. (See also
\cite[Lem.~3.7]{BruKle} to identify $Z_N(u)$ with the above definition).

\subsection{Admissible modules and $\ell$-weight
  spaces}\label{subsec:adm}

Let $\mathfrak c$ be the abelian Lie algebra generated by $D_1^{(1)}$,
\dots, $D_n^{(1)}$. Let $\ve_1$, \dots, $\ve_n$ be the basis of
$\mathfrak c^*$ dual to $D_1^{(1)}$, \dots, $D_n^{(1)}$. Let $\le$ be
the dominance order on $\mathfrak c^*$ defined by
\begin{equation*}
  \alpha\ge \beta \Longleftrightarrow 
  \alpha - \beta \in \sum_i \Z_{\ge 0} (\ve_i - \ve_{i+1}).
\end{equation*}

A $Y_n(\sigma)$-module $M$ is said to be {\it admissible\/} if we have
a generalized weight space decomposition $M =
\bigoplus_{\alpha\in\mathfrak c^*} M_\alpha$, where
\begin{equation*}
  M_\alpha = \left\{ m\in M \,\left|
  \text{ $(D_i^{(1)} - \langle \alpha, D_i^{(1)}\rangle)^N m = 0$
  for $i\in \tI$ and sufficiently large $N$}\right\}\right.
\end{equation*}
satisfying the conditions 
\begin{enumerate}
\item each $M_\alpha$ is finite-dimensional,
\item the set of all $\alpha\in\mathfrak c^*$ such that $M_\alpha\neq
  0$ is contained in a finite union of sets of the form $D(\beta) = \{
  \alpha\in\mathfrak c^* \mid \alpha\le \beta\}$ for
  $\beta\in\mathfrak c^*$.
\end{enumerate}

Since $D_i^{(r)}$ ($i\in\tI$, $r > 0$) form a commutative subalgebra
of $Y_n(\sigma)$, an admissible module $M$ has a finer decomposition
\(
   M = \bigoplus M_{\Psi}
\)
where
\begin{equation*}
  \begin{split}
  & \Psi = (\Psi_1(u),\dots, \Psi_n(u)), \quad
  \Psi_i(u) = 1 + \sum_{r=1}^\infty \Psi_i^{(r)} u^{-r}
  \in 1 + u^{-1}\C[[u^{-1}]],
\\
  & M_{\Psi} = \left\{ m\in M \,\left|
      \text{ $(D_i^{(r)} - \Psi_i^{(r)})^N m = 0$
        for $i\in \tI$, $r > 0$ and sufficiently large $N$}\right\}\right..
  \end{split}
\end{equation*}
We call $M_{\Psi}$ the {\it $\ell$-weight space\/} following
\cite[\S1.3]{Na-qaff} (or the Gelfand-Tsetlin subspace following
\cite[\S5]{BruKle}) with the $\ell$-weight $\Psi$.

It turns out that it is more appropriate to consider another sequence
of generators instead of $D_i^{(r)}$ for the $W$-algebra. We introduce
\begin{equation}\label{eq:sfA}
  \mathsf A_i(u) \defeq u^{p_1} (u-1)^{p_2}\cdots
  (u-(i-1))^{p_i} C_i(u)
\end{equation}
following \cite{FMO,BFFR}. Then it is a polynomial in $u$ in $W(\pi)$
\cite[Lem.~2.1]{FMO}. Instead of $D_i^{(r)}$, we can define an
$\ell$-weight space as a simultaneous generalized eigenspace for
operators $\mathsf A_i(u)$ and the corresponding $\ell$-weight as
an $\tI$-tuple of rational functions 
\(
   {Q} = (Q_1(u),Q_2(u),\dots, Q_n(u))
\)
as generalized eigenvalues of $\mathsf A_i(u)/\mathsf A_{i-1}(u)$
($i\in \tI$), where $\mathsf A_{-1}(u) = 1$.

Following \cite[Def.~2.8]{MR2144973} we say an $\ell$-weight ${Q}$ is
{\it $\ell$-dominant\/} if every $Q_i(u)$ is a polynomial.

\subsection{Gelfand-Tsetlin character}

The generating function of dimensions of $\ell$-weight spaces of an
admissible $Y_n(\sigma)$-module $M$ is called the {\it Gelfand-Tsetlin
  character\/} after \cite[\S5.2]{BruKle}:
\begin{equation*}
  \ch M \defeq \sum_{{\Psi}} \dim M_{{\Psi}}\, [{\Psi}].
\end{equation*}
This is an element of the completed group algebra, consisting of
formal sums satisfying the conditions corresponding to the
admissibility. See \cite[\S5.2]{BruKle} for detail.

For a $W(\pi)$-module $M$, we use the second convention for
$\ell$-weights and write
\begin{equation*}
  \ch M = \sum_{{Q}} \dim M_{{Q}}\, [{Q}].
\end{equation*}

This definition was motivated by $q$--characters of finite dimensional
representations of quantum affine algebras, introduced in
\cite{Knight,Fre-Res}, and further studied in
\cite[\S13.5]{Na-qaff}. Later in \subsecref{subsec:gr} we will show
that this is not just an analogy: both Gelfand-Tsetlin and
$q$--characters are expressed in terms of a common geometric object, a
graded quiver variety.

\subsection{$\ell$-highest weight module}\label{subsec:hw}

A vector $v$ in a $Y_n(\sigma)$-module $M$ is called an {\it
  $\ell$-highest weight vector\/} with $\ell$-highest weight
${\Psi}$ if
\begin{enumerate}
\item $E_i^{(r)} v = 0$ for $i\in I$, $r > s_{i,i+1}$,
\item $D_i^{(r)} v = \Psi_i^{(r)} v$ for $i\in\tI$, $r > 0$, 
\end{enumerate}
where $\Psi_i^{(r)}$ is a coefficient of $\Psi_i(u)$ as above.

An {\it $\ell$-highest weight module\/} is a $Y_n(\sigma)$-module $M$
generated by an $\ell$-highest weight vector.

From the triangular decomposition of $Y_n(\sigma)$, we can construct
the universal $\ell$-highest weight module $M(\sigma,{\Psi})$
with $\ell$-highest weight ${\Psi}$ \cite[\S5]{BruKle}.

We consider the quotient
\begin{equation*}
   W(\pi)\otimes_{Y_n(\sigma)} M(\sigma,{\Psi}).
\end{equation*}
By \cite[Th.~6.1]{BruKle}, it is nonzero if and only if $u^{p_i}
\Psi_i(u)\in \C[u]$ for $i\in\tI$. When it is nonzero, we call it a
{\it Verma module}. We set $P_i(u) = (u-i+1)^{p_i} \Psi_i(u-i+1)$ and
call ${P} = (P_1(u),\dots,P_n(u))$ the {\it Drinfeld polynomial\/}
of the Verma module $M({P}) \defeq W(\pi)\otimes_{Y_n(\sigma)}
M(\sigma,{\Psi})$. It has the unique simple quotient, which is
denoted by $L({P})$. It is known that both $M({P})$ and
$L({P})$ are admissible \cite[\S6.1]{BruKle}.

This convention is compatible with our definition of
$\ell$-weights. The $\ell$-highest weight vector $v$ as above has the
$\ell$-weight ${P}$ in the above second sense. Note that it is
$\ell$-dominant since every $P_i(u)$ is a polynomial.

\begin{NB}
If $v$ is an $\ell$-highest weight vector of a $W(\sigma)$-module, there
exists polynomials $P_i(u)\in \C[u]$ of degree $p_i$ such that
\begin{equation*}
  \begin{split}
    D_1(u) v &= u^{-\deg P_1} P_1(u) v,\\
    D_2(u) v &= u^{-\deg P_2} P_2(u+1) v,\\
    & \ \  \vdots \\
    D_n(u) v &= u^{-\deg P_n} P_n(u+n-1) v. 
  \end{split}
\end{equation*}
We call the $\tI$-tuple of polynomials $\vec{P}(u) =
(P_i(u))_{i\in\tI}$ an {\it $\ell$-highest weight vector}. This is the
same convention as \cite[\S6.1]{BruKle}.
\end{NB}

\section{Handsaw quiver varieties}\label{sec:handsaw}

In this section, we introduce handsaw quiver varieties following
\cite{FR}, and explain various their properties. Handsaw quiver
varieties are isomorphic to Laumon spaces (see \secref{sec:Laumon}),
and most of results are known in the context of Laumon spaces. But we
present them in terms of handsaw quivers so that they can be used in
later sections. Proofs are usually the same as ones of {\it
  ordinary\/} quiver varieties and are not given. The readers are
supposed to be familiar with \cite[\S3]{Na-alg}.

\subsection{Definition}\label{subsec:def}

Let $n\in \Z_{>0}$ be as in the previous section.
Recall $I = \{ 1, 2, \dots, n-1\}$, $\tI = I\sqcup \{n\}$.

We take $I$ and $\tI$-graded vector spaces $V = \bigoplus_{i\in I}
V_i$, $W = \bigoplus_{i\in \tI} W_i$ and consider the representation
of the following handsaw quiver
\begin{equation}\label{eq:handsaw}
  \xymatrix{
V_1 \ar@(ur,ul)_{B_2} \ar[r]^{B_1} \ar[dr]_{b}
& V_2 \ar@(ur,ul)_{B_2} \ar[dr]_{b} \ar[r]^{B_1} && \cdots 
& \ar[dr]_b \ar[r]^{B_1} & V_{n-1} \ar@(ur,ul)_{B_2} \ar[dr]_b \\
W_1 \ar[u]^a & W_2 \ar[u]^a && \cdots && W_{n-1}\ar[u]^a & W_n}
\end{equation}
where
\begin{gather*}
   B_1 \in \bigoplus_{i=1}^{n-2} \Hom(V_i,V_{i+1}), \quad
   B_2 \in \bigoplus_{i=1}^{n-1} \End(V_i), \\
   a\in\bigoplus_{i=1}^{n-1} \Hom(W_i,V_i), \quad
   b\in\bigoplus_{i=1}^{n-1} \Hom(V_i,W_{i+1}).
\end{gather*}
We denote by $\bM$ the vector space of all quadruples $(B_1,B_2,a,b)$
as above. When we want to emphasize graded dimensions of $V$, $W$, we
denote it by $\bM(\bv,\bw)$, where $\bv\in\Z[I]$, $\bw\in\Z[\tI]$ are
{\it dimension vectors\/} given by $\bv = \sum_{i\in I} \dim V_i\,
\alpha_i$, $\bw = \sum_{i\in\tI} \dim W_i\, \alpha_i$, where
$\alpha_i$ is the $i^{\mathrm{th}}$ coordinate vector of $\Z[I]$ or
$\Z[\tI]$.

In a later relation to the finite $W$-algebra $W(\pi)$, we have
\begin{equation*}
  \dim W_i = p_i
\end{equation*}

We introduce the following notation:
\begin{equation*}
  \bL(W,V) \defeq \bigoplus_i \Hom(W_i,V_i), \quad
  \bE(V,W) \defeq \bigoplus_i \Hom(V_i,W_{i+1}).
\end{equation*}
We also write $\bigoplus_i \Hom(V_i, V_{i+1})$ and
$\bigoplus_i \End(V_i)$ as $\bL(V,V)$, $\bE(V,V)$ respectively by
setting $V_n = 0$.

We define 
\begin{equation*}
   \mu(B_1,B_2,a,b) = [B_1,B_2]+ab \in \bE(V,V),
\end{equation*}
and consider an affine variety $\mu^{-1}(0)$. It is acted by the group
$G = \prod_{i\in I} \GL(V_i)$.

In order to consider quotients of $\mu^{-1}(0)$ by $G$, we introduce
the following conditions coming from the geometric invariant theory
(cf.\ \cite[\S3.ii]{Na-alg}):
\begin{Definition}
  \textup{(1)} A point $(B_1,B_2,a,b)\in\bM$ is called {\it stable\/}
  if there is no proper $I$-graded subspace $S = \bigoplus S_i$ of $V$
  stable under $B_1$, $B_2$ and containing $a(W)$.

  \textup{(2)} A point $(B_1,B_2,a,b)\in\bM$ is called {\it
    costable\/} if there is no nonzero $I$-graded subspace $S$ of $V$
  stable under $B_1$, $B_2$ and contained in $\Ker b$.
\end{Definition}

These conditions are invariant under the $G$-action. We set
\begin{equation*}
  \begin{split}
    & \Q \equiv \Q(\bv,\bw) \defeq \{ (B_1,B_2,a,b)\in\mu^{-1}(0)\mid
    \text{stable}\}/G,
    \\
    & \Q_0 \equiv \Q_0(\bv,\bw) \defeq \mu^{-1}(0)\dslash G,
    \\
    & \Q_0^{\reg}\equiv \Q_0^{\reg}(\bv,\bw) \defeq \{
    (B_1,B_2,a,b)\in\mu^{-1}(0)\mid \text{stable and
      costable}\}/G,
  \end{split}
\end{equation*}
where $\dslash$ denote the affine quotient. When we want to emphasize
$\bv$, $\bw$ as we will do later, we use the notation $\Q(\bv,\bw)$,
and $\Q$ otherwise.
These are called {\it handsaw quiver varieties}.

For a stable $(B_1,B_2,a,b)\in\mu^{-1}(0)$, the corresponding point in
$\Q$ is denoted by $[B_1,B_2,a,b]$.
A (closed) point in $\Q_0(\bv,\bw)$ is represented by a closed
$G$-orbit in $\mu^{-1}(0)$. When $(B_1,B_2,a,b)$ has a closed orbit,
the corresponding point in $\Q_0$ is denoted by $[B_1,B_2,a,b]$.

We have a projective morphism $\pi\colon \Q\to \Q_0$. The second space
$\Q_0^\reg$ is a (possibly empty) open subscheme in both $\Q$ and $\Q_0$
so that $\pi$ is an isomorphism between them.

We denote $\pi^{-1}(0)$ by $\sO$, where $0$ is the closed orbit
consisting of $(B_1,B_2,a,b) = (0,0,0,0)$. This is a projective
variety, and will played an important role later.

We have an action of $\C$ on $\bM$ by the translation of $B_2$:
$B_2\mapsto B_2 - \sum z \id_{V_i}$ ($z\in\C$). It preserves the
equation $\mu = 0$ and commutes with the $G$-action. Therefore we have
an induced $\C$-action on $\Q$, $\Q_0$, $\Q_0^\reg$. We can normalize
a point so that $\tr B_2 = 0$. Therefore we have a decomposition $\Q =
\Q^0\times\C$, etc, where $\Q^0$ is the space of normalized points.

\subsection{$\C^*$-equivariant sheaves on the projective plane}
\label{subsec:Jordan}

In this subsection, we describe $\Q$ as a fixed point locus of the
framed moduli space of torsion free sheaves $(E,\varphi)$ on $\proj^2$
of rank $r$ with $c_2(E) = n$ with respect to a $\C^*$-action, in
other words an ordinary quiver variety of Jordan type
\cite[Ch.~2]{Lecture}. The quiver variety in question corresponds to
the original ADHM description of instantons on $S^4$ \cite{ADHM}.

Let $V$, $W$ be complex vector spaces of dimension $n$, $r$
respectively.  Let
\begin{equation*}
  \widetilde\bM \defeq \End(V)\oplus \End(V) \oplus \Hom(W,V)\oplus \Hom(V,W).
\end{equation*}
An element in $\widetilde\bM$ is a quadruple $(B_1,B_2,a,b)$ according
to the above decomposition. The group $\GL(V)$ acts naturally on
$\widetilde\bM$.
We define $\mu(B_1,B_2,a,b) = [B_1,B_2]+ab$ as before.

We define the stability and costability conditions exactly as above,
where $S\subset V$ is an arbitrary subspace, not necessarily graded.

These conditions are invariant under the $\GL(V)$-action. We set
\begin{equation*}
  \begin{split}
  & M(n,r) \defeq \{ (B_1,B_2,a,b)\in\mu^{-1}(0)\mid \text{stable}\}/\!\GL(V),
\\
  & M_0(n,r) \defeq \mu^{-1}(0)\dslash\! \GL(V),
\\
  & M_0^{\reg}(n,r) 
  \defeq \{ (B_1,B_2,a,b)\in\mu^{-1}(0)\mid
  \text{stable and costable}\}/\!\GL(V).
  \end{split}
\end{equation*}

The framed moduli space of locally free sheaves is $M^{\reg}(n,r)$,
and $M(n,r)$, $M_0(n,r)$ are its Gieseker and Uhlenbeck partial
compactification respectively.

We choose and fix a homomorphism $\rho_W\colon \C^* \to \GL(W)$.
We define a $\C^*$-action on $\widetilde\bM$ by
\begin{equation}\label{eq:action}
  (B_1,B_2,a,b) \mapsto
  (t_1 B_1,B_2, a\rho_W(t_1)^{-1} , t_1\rho_W(t_1)b) \quad t_1\in\C^*.
\end{equation}
The equation $\mu(B_1,B_2,a,b) = 0$ and the (co)stability are
preserved under the $\C^*$-action, and hence we have an induced
$\C^*$-actions on $M(n,r)$, $M_0(n,r)$, $M^{\reg}(n,r)$.

Take a point $x\in M(n,r)$ and its representative
$(B_1,B_2,a,b)$. Then $x$ is fixed by the $\C^*$-action if and only if
there exists a $\rho(t_1)\in \GL(V)$ such that
\begin{equation}\label{eq:fixed}
  (t_1B_1,B_2,a\rho_W(t_1)^{-1},t_1\rho_W(t_1)b) = 
  (\rho(t_1)^{-1}B_1 \rho(t_1), \rho(t_1)^{-1}B_2 \rho(t_1),
  \rho(t_1)^{-1}a, b \rho(t_1))
\end{equation}
for any $t_1\in\C^*$. By the freeness of the $\GL(V)$-action on the
stable locus, $\rho(t_1)$ is uniquely determined by $t_1$. In
particular, the map $t_1\mapsto \rho(t_1)$ is a homomorphism. Its
conjugacy class is independent of the choice of the representative of
$x$.

We decompose $V$, $W$ as $\bigoplus V_i$, $\bigoplus W_i$ where
$t_1\in\C^*$ acts by $t_1^i$ on $V_i$, $W_i$. Then \eqref{eq:fixed}
means that $(B_1,B_2,a,b)$ is a representation of the handsaw quiver
\eqref{eq:handsaw}, where we shift the index $i$ suitably so that it
runs from $1$ to $n$.
Thus $\Q$ is a fixed point locus in $M(n,r)$ with a given conjugacy
class of a homomorphism $\rho\colon \C^*\to \GL(V)$. The same is true
for $\Q^\reg$.

We do not compare $\Q_0$ with $M_0(n,r)^{\C^*}$ directly, as we are
only intereseted in the image of $\Q$ under $\pi$ in practice.

This relation allows us to apply results on $M(n,r)$, $M_0^\reg(n,r)$
to $\Q$, $\Q_0^\reg$.
For example, $\Q$ is smooth as $M(n,r)$ is so. And the action of $G$
on the open locus consisting of stable points in $\mu^{-1}(0)$ is
free, and the quotient by $G$ is a principal $G$-bundle, etc.

\begin{Remark}\label{rem:typeA}
  A quiver variety of type $A_n$ \cite{Na-quiver}, which is associated
  with a quiver
  \begin{equation*}
    \xymatrix{
      V_1 \ar@<-.5ex>[r]_{B_1} \ar@<-.5ex>[d]_{b}
      & V_2 \ar@<-.5ex>[l]_{B_2} \ar@<-.5ex>[d]_{b} \ar@<-.5ex>[r]_{B_1} &
      \ar@<-.5ex>[l]_{B_2} & \cdots 
      & \ar@<-.5ex>[r]_{B_1} & V_{n} \ar@<-.5ex>[l]_{B_2} \ar@<-.5ex>[d]_b \\
      W_1 \ar@<-.5ex>[u]_a & W_2 \ar@<-.5ex>[u]_a && \cdots &
      & W_{n}\ar@<-.5ex>[u]_a }
  \end{equation*}
  is also a fixed component of $M(n,r)$ with respect to a
  $\C^*$-action, but a different one:
\begin{equation*}
  (B_1,B_2,a,b) \mapsto
  (t_1 B_1, t_1^{-1} B_2, a\rho_W(t_1)^{-1} , \rho_W(t_1)b) \quad t_1\in\C^*.
\end{equation*}

Note also that we exchange the stability and costability from those in
\cite{Na-quiver,Na-alg}. This change is not essential, as they are
exchanged if we replace all linear maps by their transposes.
\end{Remark}

\begin{Remark}
  In \cite{FR} a {\it chainsaw quiver variety\/} is introduced as a
  fixed point locus of $\mathfrak M(n,r)$ with respect to the cyclic
  group action defined by the same formula as \eqref{eq:action} for a
  homomorphism $\rho_W\colon \Z/n \Z\to \GL(W)$. Many of results in this section and those in \secref{sec:fixed} are generalizations to handsaw quiver varieties.
\end{Remark}

\subsection{Tangent space}

We consider the complex
\begin{NB}
\begin{equation}\label{eq:cpx'}
  \bigoplus_i \End(V_i)
  \overset{\sigma}{\longrightarrow}
  \bigoplus_i\left(
        \begin{matrix} \Hom(V_i, V_{i+1}) \oplus \End(V_i) \\ 
                        \oplus \\
                       \Hom(W_i, V_i) \oplus \Hom(V_i, W_{i+1})
                       \end{matrix}\right)
   \overset{\tau}{\longrightarrow} \bigoplus_i\Hom (V_i,V_{i+1}),
\end{equation}
\end{NB}%
\begin{equation}
  \label{eq:cpx}
  \bL(V,V)
  \overset{\sigma}{\longrightarrow}
     \begin{matrix} \bE(V, V) \oplus \bL(V,V) \\ 
                        \oplus \\
                       \bL(W, V) \oplus \bE(V, W)
                       \end{matrix}
   \overset{\tau}{\longrightarrow} \bE(V,V),
\end{equation}
where $\sigma$ and $\tau$ are defined by 
\begin{equation*}
   \sigma(\xi) = \begin{pmatrix} \xi B_1 - B_1 \xi \\
                            \xi B_2 - B_2 \xi \\
                            \xi a \\
                            - b\xi \end{pmatrix}, \quad
   \tau \begin{pmatrix} C_1 \\ C_2 \\ A \\ B \end{pmatrix}
        = [B_1,C_2] + [C_1,B_2] + aB + Ab.
\end{equation*}
This $\sigma$ is the differential of $G$-action and $\tau$ is the
differential of the map $\mu$. It is the $\C^*$-fixed part of the
complex considered in \cite[(2.13)]{MR2095899}.
It follows that $\sigma$ is injective and $\tau$ is surjective, and the
tangent space of $\Q$ at $[B_1,B_2,a,b]$ is isomorphic
to the middle cohomology group of the above complex.

As an application we have
\begin{equation*}
  \dim \Q = \dim \bL(W,V) + \dim \bE(V,W)
  = \sum_i \dim V_i (\dim W_i + \dim W_{i+1}).
\end{equation*}

As we remarked above, we have a natural principal $G$-bundle over $\Q$
from the construction. Since $V_i$ is a representation of $G$, we have
an induced vector bundle over $\Q$, which we also denote by $V_i$ for
brevity. We also consider $W_i$ as a trivial bundle. Then $\bL(V,V)$,
$\bE(V,V)$, etc are vector bundles over $\Q$, and $\tau$, $\sigma$ are
vector bundle homomorphisms. Therefore $\Ker\tau/\Ima\sigma$ is also a
vector bundle, which is isomorphic to the tangent bundle of $\Q$.

\subsection{Stratification}\label{subsec:strat}

According to a local structure at a closed orbit $x = [B_1,B_2,a,b]$,
we have a natural stratification of $\Q_0(\bv,\bw)$, similar to one for
a quiver variety \cite[\S6]{Na-quiver}.

As a first step of the stratification, we have $V = V' \oplus V''$ such
that the stabilizer subgroup acts trivially on $V''$. The restriction
of $(B_1,B_2,a,b)$ to $V''$ defines a point in $\Q_0^\reg(\bv'',\bw)$,
where $\bv''$ is the dimension vector of $V''$. The restriction of
$(B_1,B_2,a,b)$ to $V'$ is a datum with $W = 0$. Thus we have $a = 0 =
b$. We take a homomorphism $\C^*\to \prod \GL(V'_i)$ given by
$t\mapsto \prod t^i \id_{V'_i}$, and consider the limit $t\mapsto
0$. Then the closedness of the orbit implies $B_2|_{V'} =
0$. Therefore we only need to consider the closed orbit
$\left.B_1\right|_{V'_i}\in \End(V'_i)$ with respect to the
$\GL(V_i)$-action. Closed orbits of $\End(V'_i)$ under the conjugation
are classified by their eigenvalues with multiplicities, and hence
points in the symmetric product $S^{\dim V'_i}\C$. We thus have
\begin{equation*}
  \Q_0(\bv,\bw) = \bigsqcup \Q_0^\reg(\bv-\bv',\bw)
  \times S^{v'_1}\C\times S^{v'_2}\C\times\cdots \times S^{v'_{n-1}}\C,
\end{equation*}
where $\bv' = (v'_1,v'_2,\dots,v'_{n-1})$. 

We denote by $S^{\bv'}\C$ the product $S^{v'_1}\C\times
S^{v'_2}\C\times\cdots \times S^{v'_{n-1}}\C$, and consider it as the
configuration space of unordered colored points, where the set of
colors is $I$.

We refine the above stratification by further decomposing the
symmetric product part $S^{\bv'}\C$ (see \cite[\S1.3]{Kuz}):
\begin{equation*}
  S^{\bv'}\C = \bigsqcup_\Gamma S^{\bv'}_\Gamma \C,
\end{equation*}
where $\Gamma$ is a collection $\{ \bv'_1,\dots, \bv'_m\}$ of $\Z_{\ge
  0}[I]\setminus \{0\}$ with the ordering disregarded, satisfying
$\bv'_1+\cdots+\bv'_m = \bv'$. Here the number $m$ is not fixed.
The corresponding stratum $S^{\bv'}_\Gamma \C$ is defined as
\begin{equation*}
  S^{\bv'}_\Gamma \C = \left.\left\{ \sum_{\alpha=1}^m \bv'_\alpha x_\alpha\,
  \right| \text{$x_\alpha \neq x_\beta$ for $\alpha\neq \beta$} \right\},
\end{equation*}
where $\bv'_\alpha x_\alpha$ is a point $x_\alpha$ whose multiplicity
for a color $i\in I$ is given by the $i^{\mathrm{th}}$-component of
$\bv'_\alpha$. We thus get
\begin{equation}\label{eq:stratum}
  \Q_0(\bv,\bw) = \bigsqcup \Q_0^\reg(\bv-\bv',\bw)
  \times S^{\bv'}_\Gamma\C.
\end{equation}
We have
\begin{equation}\label{eq:dimofstratum}
  \dim \Q_0^\reg(\bv-\bv',\bw) \times S^{\bv'}_\Gamma\C = 
  \sum_i (\dim V_i - \dim V'_i) (\dim W_i + \dim W_{i+1}) + m.
\end{equation}

From \cite[\S3.2]{Na-qaff}, which is applicable thanks to
\subsecref{subsec:Jordan}, a local structure of a neighborhood of
$x\in\Q_0(\bv,\bw)$ and $\pi^{-1}(x)$ is determined by the stratum
$\Q_0^\reg(\bv-\bv',\bw)\times S^{\bv'}_\Gamma\C$ to which $x$ is
contained. Moreover it is described as a local structure of a
$\C^*$-fixed subvariety of products of several copies of quiver
varieties of Jordan type around $0$ and $\pi^{-1}(0)$.
As a particular result, we find that $\pi^{-1}(x)$ is isomorphic to
the product
\begin{equation*}
  \sO(\bv'_1,\bw)\times\cdots\times \sO(\bv'_m,\bw)
\end{equation*}
of various central fibers. This was observed in
\cite[Prop.~2.1.2]{Kuz} in terms of Laumon spaces.

\subsection{Group action}\label{subsec:group}

Let $G_\bw = \prod_{i\in \tI} \GL(W_i)$. It acts on $\bM$ by
conjugation. It preserves the equation $\mu = 0$ and commutes with the
$G$-action which we have used to define quotients. Therefore we have
induced $G_\bw$-actions on $\Q$ and $\Q_0$.

We also have a $\C^*$-action induced from
\begin{equation*}
  (B_1,B_2,a,b) \mapsto (B_1, t B_2, a, tb).
\end{equation*}
We set $\tG_\bw = \C^*\times G_\bw$.

\begin{NB}
Let $A$ be an abelian reductive subgroup of $\tG_\bw$. Suppose
$x\in\Q$ is fixed by $A$. Taking a representative $(B_1,B_2,a,b)$ of
$x$, we have a homomorphism $\rho\colon A\to G$ by the same argument
in \subsecref{subsec:Jordan}. Its conjugacy class is independent of
the choice of a representative of $x$. We denote by $Q(\rho)$ the
subvariety consisting of $A$-fixed points with given $\rho$ up to
conjugacy. This is a union of connected components of $Q(\rho)^A$, and
hence smooth.
\end{NB}

\section{Based maps and Laumon space}\label{sec:Laumon}

Let $\proj^1$ be the projective line and $z$ its inhomogeneous
coordinate.

The handsaw variety $\Q$ is isomorphic to the parabolic Laumon space,
that is the moduli space of flags
\begin{equation*}
  0 = E_0 \subset E_1 \subset \cdots \subset E_{n-1} 
  \subset E_n = W\otimes \shfO_{\proj^1}
\end{equation*}
of locally free sheaves over $\proj^1$ such that $\rank E_i =
\sum_{j\le i} \dim W_j$, $\langle c_1(E_i),[\proj^1]\rangle = -\dim
V_i$, and at $z=\infty\in\proj^1$ the infinity of $\proj^1$ equal to
$0\subset W_1 \subset W_1\oplus W_2 \subset \cdots\subset W_1\oplus
W_2\oplus\cdots \oplus W_{n-1} \subset W$.
We omit the word `parabolic' from the Laumon space for brevity
hereafter.

Remark that the inclusion $E_i\subset E_{i+1}$ is an injective
homomorphism of sheaves, and $E_i$ is not necessarily a subbundle of
$E_{i+1}$, where a stronger condition that $E_{i+1}/E_i$ is locally
free is required. If we impose this stronger condition, we get the
space of based rational maps from $\proj^1$ to the partial flag
variety of $W$. Here {\it based\/} means that $\infty$ is mapped to
the specific point $0\subset W_1 \subset W_1\oplus W_2 \subset
\cdots\subset W_1\oplus W_2\oplus\cdots \oplus W_{n-1} \subset W$.
This (possibly empty) open subscheme of Laumon space is isomorphic to
$\Q^{\reg}$.

This result is proved in \cite[\S2.3]{FR}, and is mentioned that it
goes back to Str\o mme \cite{Stromme}. It also follows from the ADHM
description (see \subsecref{subsec:Jordan}) together with Atiyah's
observation \cite{Atiyah} that the space of based rational maps is a
$\C^*$-fixed point in a framed instanton moduli space by the
observation in \subsecref{subsec:Jordan}. Since these proofs are
combination of results scattered in various literature, we give a
self-contained argument in this subsection for the sake of a reader.

It is also possible to describe $\pi\colon \Q\to \Q_0$ as
follows. Consider the inclusion $E_i\subset
W\otimes\shfO_{\proj^1}$. The quotient $W\otimes\shfO_{\proj^1}/E_i$
is not necessarily locally free, so we take its locally free part
$(W\otimes\shfO_{\proj^1}/E_i)_{l.f.}$ and define
\begin{equation*}
  \tilde E_i = \Ker \left[W\otimes\shfO_{\proj^1}\to 
  (W\otimes\shfO_{\proj^1}/E_i)_{l.f.}\right].
\end{equation*}
Then we have $E_i\subset \tilde E_i\subset W\otimes\shfO_{\proj^1}$ and
$W\otimes\shfO_{\proj^1}/\tilde E_i$ is locally free. We have
\(
  \tilde E_i \subset \tilde E_{i+1}
\)
and $\tilde E_{i+1}/\tilde E_i$ is locally free from the construction.
Therefore $0 = \tilde E_0\subset \tilde E_1\subset \cdots\subset E_n$
defines a point in $\Q^\reg$. The length of $\tilde E_i/E_i$ defines a
point in the symmetric product $S^n\C$ where $n = \langle c_1(\tilde
E_i) - c_1(E_i), [\proj^1]\rangle$. The identification of this
construction and $\pi$ is left as an exercise for the reader.

\subsection{From a handsaw quiver variety to a Laumon space}
\label{subsec:constr}

For $i \in I$ we consider homomorphisms
\begin{equation}
  \label{eq:alpha}
\begin{gathered}
  \alpha_i \colon
  \left(W_1 \oplus W_2 \oplus \cdots\oplus W_i\oplus V_i\right)
    \otimes\shfO_{\proj^1}
  \to V_i\otimes \shfO_{\proj^1}(1),
\\
  \beta_i \colon
  \left(W_1 \oplus W_2 \oplus \cdots\oplus W_i\oplus V_i\right)
    \otimes\shfO_{\proj^1} \to 
  \left(W_1 \oplus W_2 \oplus \cdots\oplus W_{i+1}\oplus V_{i+1}\right)
    \otimes\shfO_{\proj^1},
\end{gathered}
\end{equation}
defined by
\begin{equation*}
   \alpha_i \defeq (B_1^{i-1}a, B_1^{i-2}a, \dots, a, z - B_2),\qquad
   \beta_i \defeq \bigoplus_{j\le i} \id_{W_j} \oplus
   \begin{pmatrix}
     b \\ B_1
   \end{pmatrix}.
\end{equation*}
The equation $\mu = 0$ is equivalent to
\(
   B_1 \alpha_i = \alpha_{i+1} \beta_i.
\)

We consider linear maps $\left.\alpha_i\right|_z$,
$\left.\beta_i\right|_z$ between fibers of vector bundles at a point
$z\in\proj^1$.

\begin{Lemma}\label{lem:st}
  \textup{(1)} $(B_1,B_2,a,b)$ is stable if and only
  if $\left.\alpha_i\right|_z$ is surjective for all $i$.

  \textup{(2)} $(B_1,B_2,a,b)$ is costable if and only if
  $\Ker\left.\alpha_i\right|_z\cap \Ker\left.\beta_i\right|_z$ is zero
  for all $i$.
\end{Lemma}

\begin{proof}
  (1) We first prove the $\Rightarrow$ direction.

  We set $\lsp{i}S_i$ as the image of $\left.\alpha_i\right|_z$
  and consider it as a subspace of $V_i$. We set all other
  $\lsp{i}S_j$ ($j\neq i$) as $V_j$, and define a graded subspace
  $\lsp{i}S = \bigoplus \lsp{i}S_j$. It contains $a(W)$. We have
  \begin{equation*}
     B_2 \alpha_i = \alpha_i
     \left(
       \begin{array}{ccc|c}
         & z\id && 0 \\ \hline
         -B_1^{i-1}a & \cdots & -a & B_2
       \end{array}
     \right).
  \end{equation*}
  This implies $\lsp{i}S$ is invariant under $B_2$.

  We show that $\lsp{i}S$ is invariant under $B_1$ by the induction on
  $i$. It is clear for $i=1$. Suppose that it is true for $i-1$. Then
  the stability of $(B_1,B_2,a,b)$ implies that $\lsp{i-1}S = V$. It
  means that $\left.\alpha_{i-1}\right|_z$ is surjective. Therefore
  $B_1 \alpha_{i-1} = \alpha_{i} \beta_{i-1}$ implies
  $B_1(V_{i-1})\subset \Ima\alpha_i$. Thus $\lsp{i}S$ is also
  invariant under $B_1$. It implies the surjectivity of
  $\left.\alpha_i\right|_z$ as we have already noticed.

  We next prove the $\Leftarrow$ direction. Suppose that
  $(B_1,B_2,a,b)$ is not stable and take a proper $I$-graded subspace
  $S = \bigoplus S_i$ of $V$ stable under $B_1$, $B_2$ and containing
  $a(W)$. We consider the annihilator $S^\perp$ of $S$ in $V^*$. Then
  it is invariant under $\lsp{t}B_1$, $\lsp{t}B_2$ and contained in
  $\Ker\lsp{t}a$. Consider a nonzero component $S_i^\perp$. Since it
  is preserved by $\lsp{t}B_2$, there exists an eigenvector $0\neq
  \phi$ of $\lsp{t}B_2$ in $S_i^\perp$. If the eigenvalue is $z$,
  $\lsp{t}\alpha_i$ has kernel at the point
  $z\in\C\subset\proj^1$. This contradicts the surjectivity of
  $\alpha_i$ at $z$. We thus have $S^\perp = 0$, i.e., $S = V$.

  (2) Note first that
  \begin{equation*}
    \Ker\left.\alpha_i\right|_z\cap \Ker\left.\beta_i\right|_z
    \cong \Ker
    \begin{bmatrix}
      z - B_2 \\ b \\ B_1
    \end{bmatrix}
    \colon V_i \to V_i\oplus W_{i+1} \oplus V_{i+1}.
  \end{equation*}

  Let us show the $\Rightarrow$ direction.
  Suppose that $\Ker\left.\alpha_i\right|_z\cap
  \Ker\left.\beta_i\right|_z$ is nonzero. It means that
  \(
    S_i \defeq \Ker( z - B_2)\cap \Ker b\cap \Ker B_1
  \)
  (in $V_i$) is nonzero.
  Setting other spaces $S_j = 0$, we define a graded subspace $S =
  \bigoplus_j S_j$. It is contained in $\Ker b$ and is invariant under
  $B_1$, $B_2$. 
  \begin{NB}
    Let us show that it is invariant also under $B_2$. Take $v\in
    S_i$. Then $B_2 v = zv$. Therefore $B_2 v$ is also in $S_i$.
  \end{NB}%
  The costability implies $S_i = 0$. This is a contradiction.

  We next prove $\Leftarrow$ direction. Suppose that $(B_1,B_2,a,b)$
  is not costable and take a nonzero $I$-graded subspace $S =
  \bigoplus S_i$ of $V$ stable under $B_1$, $B_2$ and contained in
  $\Ker b$. We prove $S_i = 0$ by an induction from above.

  Suppose $S_{n-1} \neq 0$. Since it is invariant under $B_2$, we have an
  eigenvector $0\neq v\in S_{n-1}$ with the eigenvalue $z$. Then
\(
  \Ker\left.\alpha_{n-1}\right|_z\cap \Ker\left.\beta_{n-1}\right|_z
\)
is nonzero. This is a contradiction, and hence we must have $S_{n-1} = 0$.

Suppose $S_j = 0$ for $j > i$ and $S_i \neq 0$. Then the same argument
as above shows
\(
   \Ker\left.\alpha_i\right|_z\cap \Ker\left.\beta_i\right|_z
\)
is nonzero. This contradicts the assumption, and hence we have
$S_i = 0$.
\end{proof}

Suppose that $(B_1,B_2,a,b)$ is stable. Thanks to the above lemma,
$E_i\defeq \Ker\alpha_i$ is a vector bundle of rank $\sum_{j\le i}
\dim W_j$ with $\langle c_1(E_i),[\proj^1]\rangle = -\dim V_i$.  Its
fiber at $z=\infty$ is identified with $W_1\oplus\cdots\oplus W_i$.

The homomorphism $\beta_i$ induces $E_i\to E_{i+1}$. From the proof of
\lemref{lem:st}(2) the corresponding linear map between fibers at a
point $z$ is not injective only when $z$ is an eigenvalue of
$B_2$. Thus it can happen only at finitely many points in
$\proj^1$. Therefore $E_i\to E_{i+1}$ is injective as a sheaf
homomorphism. We thus get a flag of locally free sheaves, that is a
point in a Laumon space. It is independent on the choice of a point in
a $G$-orbit, and hence a map from $\Q$ to the Laumon space. It also
gives a map from $\Q^\reg$ to the space of based maps by
\lemref{lem:st}(2).
We have an universal family on the space $\Q$, and therefore the above
construction gives a morphism from $\Q$ to the Laumon space.

\subsection{From a Laumon space to a handsaw quiver variety}

Let us describe the inverse morphism.

From the definition of $E_i$, we have a short exact sequence
\begin{equation*}
  0 \to E_i \to (W_1\oplus\cdots \oplus W_i\oplus V_i)\otimes\shfO_{\proj^1}
  \xrightarrow{\alpha_i} V_i\otimes\shfO_{\proj^1}(1)\to 0.
\end{equation*}
We thus have
\begin{equation*}
  V_i \cong H^1(\proj^1, E_i(-1)).
\end{equation*}
Using this observation, we take this as a definition of $V_i$ for the
inverse morphism. Note that $H^0(\proj^1,E_i(-1)) = 0$ as $E_i$ is a
subsheaf of a trivial sheaf $E_n = W\otimes\shfO_{\proj^1}$. Hence
$\dim V_i$ is $-\langle c_1(E_i),\proj^1\rangle$ as expected.

We next study $H^1(\proj^1,E_i(-2))$ and show that it is isomorphic to
$W_1\oplus\cdots \oplus W_i\oplus V_i$. This is expected from the same
short exact sequence.

Considering a short exact sequence
\begin{equation*}
  0 \to E_i(-2) \to E_i(-1) \to E_i|_{z=\infty}\to 0,
\end{equation*}
we get
\begin{equation}\label{eq:E_i(-2)}
  0\to W_1\oplus\cdots\oplus W_i \xrightarrow{\varphi} H^1(\proj^1,E_i(-2))
  \xrightarrow{\psi} V_i = H^1(\proj^1,E_i(-1)) \to 0.
\end{equation}

Since $E_i$ is a subsheaf of $W\otimes\shfO_{\proj^1}$, we have a
homomorphism
\begin{equation*}
  H^1(\proj^1,E_i(-2)) \to
  W\otimes H^1(\proj^1,\shfO_{\proj^1}(-2))
  \cong W.
\end{equation*}
When we compose it with $\varphi$, it gives the inclusion
\begin{equation*}
  W_1\oplus\cdots\oplus W_i \to W.
\end{equation*}
\begin{NB}
This follows from
\begin{equation*}
  \begin{CD}
    0 @>>> E_i(-2) @>>> E_i(-1) @>>> W_1\oplus\cdots\oplus W_i @>>> 0
\\
    @. @VVV @VVV @VVV @.
\\
    0 @>>> W\otimes\shfO_{\proj^1}(-2) @>>> W\otimes\shfO_{\proj^1}(-1)
    @>>> W @>>> 0.
  \end{CD}
\end{equation*}
\end{NB}%
Therefore we compose this with the projection $W\to W_1\oplus\cdots
\oplus W_i$ to get a splitting of \eqref{eq:E_i(-2)}. Therefore we
have
\begin{equation*}
  H^1(\proj^1, E_i(-2)) \cong W_1\oplus\cdots \oplus W_i\oplus V_i.
\end{equation*}

We define $B_1\colon V_i\to V_{i+1}$ as an induced homomorphism
\begin{equation*}
  H^1(\proj^1,E_i(-1)) \to H^1(\proj^1,E_{i+1}(-1))
\end{equation*}
from the inclusion $E_i\to E_{i+1}$. 

The multiplication of the inhomogeneous coordinate $z$ induces a
homomorphism
\begin{equation*}
  W_1\oplus \cdots \oplus W_i \oplus V_i = H^1(\proj^1,E_i(-2))
  \xrightarrow{\zeta} V_i = H^1(\proj^1,E_i(-1)).
\end{equation*}
We set its restriction to $W_i$ and $V_i$ as $a$ and $-B_2$
respectively. Other components for $W_1$, \dots, $W_{i-1}$ come from
the same homomorphisms for $E_{i-1}$ together with $B_1$. Then by
induction we see that they are $B_1^{i-1}a$, \dots, $B_1 a$.
In view of the construction in \subsecref{subsec:constr} this is
natural as $\zeta = \left.\alpha_i\right|_{z=0}$.
In fact, $E_i$ is recovered as
\begin{equation*}
  \Ker\left(\zeta - z\psi \colon
    H^1(\proj^1,E_i(-2))\otimes \shfO_{\proj^1}(-1)
    \to
    H^1(\proj^1,E_i(-1))\otimes \shfO_{\proj^1}\right),
\end{equation*}
as one can show using the resolution of the diagonal
\begin{equation*}
  0 \to \shfO_{\proj^1}(-1)\boxtimes\shfO_{\proj^1}(-1)
  \to \shfO_{\proj^1\times\proj^1} \to \shfO_{\Delta} \to 0.
\end{equation*}
(This is a basis of the well-known theorem: the derived category of
coherent sheaves on $\proj^1$ is equivalent to that of representations
of the Kronecker quiver.) The above definition has been given so that
$\zeta - z\psi$ coincides with $\alpha_i$ in the previous subsection.

We next consider the commutative diagram
\begin{equation*}
  \begin{CD}
  H^1(\proj^1,E_i(-2)) @>>> H^1(\proj^1,E_{i+1}(-2))
\\
  @V\cong VV @VV\cong V
\\
  W_1\oplus\cdots\oplus W_i\oplus V_i
  @>>> W_1\oplus\cdots\oplus W_i\oplus W_{i+1} \oplus V_{i+1},
  \end{CD}
\end{equation*}
where the upper arrow is induced homomorphism from the inclusion
$E_i\to E_{i+1}$. The restriction of the lower arrow to
$W_1\oplus\cdots\oplus W_i$ is the inclusion to $W_1\oplus\cdots\oplus
W_i\oplus W_{i+1}$. And the component from $V_i\to V_{i+1}$ is given
by $B_1$. These come from consideration of \eqref{eq:E_i(-2)} and the
same sequence for $E_{i+1}$. Let us set the remaining component
$V_i\to W_{i+1}$ as $b$.

We have just defined all linear maps $B_1$, $B_2$, $a$, $b$. They
satisfy the equation $\mu = 0$. This follows from the commutativity of
the diagram
\begin{equation*}
  \begin{CD}
    H^1(\proj^1,E_i(-2)) @>>> H^1(\proj^1,E_{i+1}(-2))
\\
   @VVV @VVV
\\
   H^1(\proj^1,E_i(-1)) @>>> H^1(\proj^1,E_{i+1}(-1)),
  \end{CD}
\end{equation*}
where the vertical arrows are given by the multiplication of $z$.
\begin{NB}
In fact, this commutative diagram gives  
\begin{equation*}
  \begin{CD}
    W_1\oplus\cdots\oplus W_i\oplus V_i
  @>>{\bigoplus_{j\le i} \id_{W_i} \oplus b \oplus B_1}> 
  W_1\oplus\cdots\oplus W_i\oplus W_{i+1} \oplus V_{i+1}
\\
   @V B_1^{i-1}a\oplus \dots a\oplus -B_2 VV
   @VV B_1^{i}a\oplus \dots a\oplus -B_2 V
\\
    V_i @>B_1 >> V_{i+1}.
  \end{CD}
\end{equation*}
This is nothing but $B_1\alpha_i = \alpha_{i+1}\beta_i$ evaluated at $z=0$.
\end{NB}

The stability of $(B_1,B_2,a,b)$ follows from \lemref{lem:st}(1), as
we have already noticed that $E_i$ is recovered as $\Ker\alpha_i$.

If the original $E_1\subset\cdots\subset E_{n}$ is a flag of
subbundles, then $(B_1,B_2,a,b)$ is costable by \lemref{lem:st}(2).

We thus have maps from the Laumon space and the based map space to
$\Q$ and $\Q^\reg$ respectively. These are morphisms thanks to a
universal family on $\Q$. They are inverse to morphisms constructed in
the previous subsection, as we have checked during the above proof.

\section{Fixed points}\label{sec:fixed}

In this section we study the fixed point set with respect to the
action of the maximal torus $\hT_\bw$ of $\tG_\bw$. One of
applications is a formula of Betti numbers of $\sO(\bv,\bw)$, which
will be used in a later section. This study is well known in the
context of Laumon spaces as in the previous section, but we present it
here for completeness.

\subsection{Parametrization}

We fix a decomposition $W_i = \bigoplus_\alpha W_i^{\alpha}$ of
$W_i$ into sum of $1$-dimensional subspaces, and consider the torus
$T_i$ of $\GL(W_i)$ preserving it. Here $\alpha$ runs from $1$ to
$\dim W_i$.
We set $\hT_\bw = \C^*\times\prod_i T_i$. This is an example of an
abelian reductive subgroup of $\tG_{\bw}$ considered in
\subsecref{subsec:group}. It is maximal among such subgroups.

Let us describe $\hT_\bw$-fixed points in $\Q$. Take a point $x\in\Q$
and its representative $(B_1,B_2,a,b)$. It is fixed by $\prod_i T_i$
if and only if we have a direct sum decomposition
\begin{equation*}
   V = \bigoplus_{i,\alpha} V^{\alpha}_i, \qquad
   V^{\alpha}_i = \bigoplus_{j\in I} V^{\alpha}_{i;j}
\end{equation*}
such that $V_i^{\alpha}$ with $W_i^{\alpha}$ is invariant under
$(B_1,B_2,a,b)$. Here $W_i^{\alpha}$ is considered as an
$\tI$-graded vector space by setting its $i^{\mathrm{th}}$-component
as $W_i^{\alpha}$ and all other components as $0$. Then each summand
of $(B_1,B_2,a,b)$ corresponding to $W_i^{\alpha}$ is a
$\C^*$-equivariant rank $1$-torsion free sheaf on $\proj^2$, in other
words, a $\C^*$-equivariant ideal sheaf on $\C^2$ by
\subsecref{subsec:Jordan}.

It is fixed further by the remaining $\C^*$ if and only if the ideal
sheaf is $\C^*\times \C^*$-equivariant, where the action is given by
\begin{equation*}
  (B_1,B_2,a,b) \mapsto (t_1 B_1, t B_2, a, t_1 t b).
\end{equation*}
It is a monomial ideal and is parametrized by a Young diagram (see
e.g., \cite[Chapter~5]{Lecture}).

In summary, $\hT_\bw$-fixed points in $\Q$ are parametrized by a tuple
of Young diagrams $Y_i^{\alpha}$ indexed by $i\in\tI$ and
$\alpha=1,\dots,\dim W_i$.
We denote by $\vec{Y}$ the tuple of Young diagrams, and consider it as
a point in $\Q$.

When we draw Young diagrams on the plane, it is natural to shift
$Y_i^{\alpha}$ to the right from the convention in
\cite[Chapter~5]{Lecture} so that the $x$-coordinate of the box at the
bottom left corner is $i$. See Figure~\ref{fig:Young}.
The constraint is that the total number of boxes whose $x$-coordinates
are $i$ is equal to $\dim V_i$.

\begin{figure}[htbp]
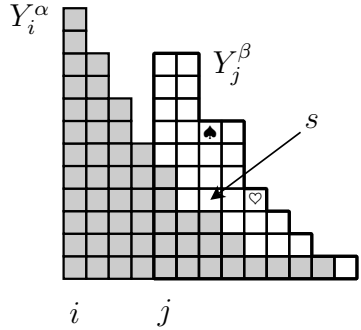

\def\JPicScale{1.5}
\centering
\input Young.pst
\caption{A tuple of Young diagrams}
\label{fig:Young}
\end{figure}

Take a box $s$ in the plane as in the figure. We define its leg length
$l_{Y_j^\beta}(s)$ and arm length $a_{Y_j^\beta}(s)$ with respect to
the Young diagram $Y_j^\beta$ as follows. We put a $\heartsuit$ mark
on the rightmost box in $Y_j^\beta$ in the same row with $s$. If there
is no row in $Y_j^\beta$, we put a mark on the box whose
$x$-coordinate is $(j-1)$. Then $l_{Y_j^\beta}(s)$ is the difference
of the $x$-coordinates of the marked box and $s$. In the above example,
it is equal to $2$. Changing rows with columns, we define the arm
length in the same way, where we put a $\spadesuit$ mark for the
relevant one in the figure. Thus we have $a_{Y_j^\beta}(s) = 3$ in the
example.
These may have negative values if $s$ is not in the Young diagram. In
the above example, the leg and arm lengths with respect to $Y_i^\alpha$ are 
\(
   l_{Y_i^\alpha}(s) = -2,
\)
\(
   a_{Y_i^\alpha}(s) = -1
\)
respectively.

\subsection{Character of the tangent space}

Let $\vec{Y}$ be a $\hT$-fixed point in $\Q$ as above. The tangent space
$T_{\vec{Y}}\Q$ of $\Q$ at $\vec{Y}$ is a $\hT$-module. 

We denote by $e^\alpha_i$ the character of $\hT\to \C^*$,
corresponding to the summand $W_i^\alpha$.

\begin{Proposition}\label{prop:tangent}
  The character of the $\hT$-module $T_{\vec{Y}}\Q$ is given by
  \begin{equation*}
    \ch T_{\vec{Y}}\Q =
    \sum_{(i,\alpha),(j,\beta)}
    e^\beta_j \left(e^\alpha_i\right)^{-1}
    \left(
    \sum_{\substack{s\in Y^\alpha_i \\ l_{Y_j^\beta}(s) = 0}} t^{a_{Y^\alpha_i}(s)+1}
    +
    \sum_{\substack{s\in Y^\beta_j \\ l_{Y_i^\alpha}(s)=-1}} t^{-a_{Y^\beta_j}(s)}
    \right).
  \end{equation*}
\end{Proposition}

This is a direct consequence of the formula of the tangent space of
$M(n,r)$ at the torus fixed point $\vec{Y}$ (see e.g.,
\cite[Th.~2.11]{MR2199008}), as $\Q$ is an union of components of the fixed
point locus $M(n,r)^{\C^*}$. More precisely, the character of the
$\C^*\times\hT$-module $T_{\vec{Y}}M(n,r)$ is given by
\begin{equation*}
  \sum_{(i,\alpha),(j,\beta)}
    e^\beta_j \left(e^\alpha_i\right)^{-1}
    \left(
    \sum_{s\in Y^\alpha_i} t_1^{-l_{Y_j^\beta}(s)} t^{a_{Y^\alpha_i}(s)+1}
    +
    \sum_{s\in Y^\beta_j} t_1^{l_{Y_i^\alpha}(s)+1} t^{-a_{Y^\beta_j}(s)}
    \right).
  \end{equation*}
And we take its coefficient of $t_1^0$ to get the above formula.

Let us explain the meaning of an individual term for $(i,\alpha)$,
$(j,\beta)$ in the sum. The complex \eqref{eq:cpx} computing the
tangent space of $\Q$ decomposes according to $W =
\bigoplus_{i,\alpha} W_i^\alpha$, $V = \bigoplus V_i^\alpha$:
\begin{equation*}
  \bL(V_i^\alpha,V_j^\beta)
  \overset{\sigma}{\longrightarrow}
     \begin{matrix}
       \bE(V_i^\alpha, V_j^\beta) \oplus \bL(V_i^\alpha,V_j^\beta) \\ 
       \oplus \\
       \bL(W_i^\alpha, V_j^\beta) \oplus \bE(V_i^\alpha, W_j^\beta)
     \end{matrix}
   \overset{\tau}{\longrightarrow} \bE(V_i^\alpha,V_j^\beta).
\end{equation*}
Let us denote this by $C_{(i,\alpha),(j,\beta)}$, where we set the
middle term to be of degree $0$. Then $\ch T_{\vec{Y}}\Q
= \sum_{(i,\alpha),(j,\beta)} \ch C_{(i,\alpha),(j,\beta)}$.
Note that we have
\begin{equation}\label{eq:C}
  \rank C_{(i,\alpha),(j,\beta)} =
  \dim V^\beta_{j;i} + \dim V^\alpha_{i;j-1}.
\end{equation}
One of two terms vanishes according to either $i\ge j$ or $i < j$.

\subsection{Betti numbers of handsaw quiver varieties}

Since $\Q$ has a torus action with isolated fixed points, we can
compute Betti numbers by decomposing it into $(\pm)$-attracting sets.

Let us choose $\lambda\colon \C^*\to \hT = \C^*\times \prod_i T_i$ so
that the weight of the $\C^*$-component is $1$ and the difference of
the weights of $W_i^\alpha$ and $W_j^\beta$-components is sufficiently
large if $(i,\alpha) > (j,\beta)$. Here we put the lexicographic order
on the set of indices $(i,\alpha)$, i.e., $(i,\alpha) > (j,\beta)$ if
and only if $i > j$ or $i=j$, $\alpha > \beta$.

Under this choice of $\lambda$, every point in $\Q_0$ converges to $0$
as $t\to 0$, and every point except $0$ goes to infinity as
$t\to\infty$ by the argument in \cite[Th.~3.5]{MR2095899}.

For each fixed point $\vec{Y}$ we consider $(\pm)$-attracting set:
\begin{equation*}
  S_{\vec{Y}} = \left\{ x\in\Q \left|\, 
      \lim_{t\to 0} \lambda(t)\cdot x = \vec{Y}
      \right.\right\},
\quad  
  U_{\vec{Y}} = \left\{ x\in\Q \left|\
      \lim_{t\to \infty} \lambda(t)\cdot x = \vec{Y}
      \right.\right\}.
\end{equation*}
By the remark above, we have $\sqcup_{\vec{Y}} S_{\vec{Y}} = \Q$ and
$\sqcup_{\vec{Y}} U_{\vec{Y}} = \sO$.
These are affine spaces whose dimensions are equal to sum of
dimensions of positive and negative weight spaces of $T_{\vec{Y}}\Q$.

Let us calculate these dimensions by using the formula in
\propref{prop:tangent}. From our choice of $\lambda$, the sum of the
dimensions of positive weight spaces is the number of terms with
$(i,\alpha) < (j,\beta)$ or $(i,\alpha) = (j,\beta)$ and the power of
$t$ is positive. If $(i,\alpha) = (j,\beta)$, only the first sum in
the parentheses survive, and hence the power is positive
always. Therefore we have
\begin{equation*}
  \dim S_{\vec{Y}} = \sum_{(i,\alpha) \le (j,\beta)} \rank C_{(i,\alpha),(j,\beta)},
  \qquad
  \dim U_{\vec{Y}} = \sum_{(i,\alpha) > (j,\beta)} \rank C_{(i,\alpha),(j,\beta)}.
\end{equation*}
By \eqref{eq:C}, we get
\begin{equation}\label{eq:attracting}
  \begin{split}
  \dim S_{\vec{Y}} &= \sum_{i} \dim V_i \dim W_{i+1} + \sum_{(i,\alpha)} \alpha \dim V^\alpha_{i;i},
\\
  \dim U_{\vec{Y}} &= \sum_i \dim V_i \dim W_i - \sum_{(i,\alpha)} \alpha \dim V^\alpha_{i;i}.
  \end{split}
\end{equation}
\begin{NB}
\begin{equation*}
  \begin{split}
  \dim S_{\vec{Y}} &= \sum_{\substack{(i,\alpha) \le (j,\beta) \\ i < j}} \dim V^\alpha_{i;j-1}
  + \sum_{i,\alpha \le \beta} \dim V^\beta_{i;i} \\
  &= \sum_{\substack{(i,\alpha)}} \sum_{j=i}^{n-1} \dim V^\alpha_{i;j}\dim W_{j+1}
  + \sum_{i,\beta} \beta \dim V^\beta_{i;i}, 
  \end{split}
\end{equation*}
\begin{equation*}
  \begin{split}
  \dim U_{\vec{Y}} &= \sum_{\substack{(i,\alpha) < (j,\beta)}} \dim V^\alpha_{i;j} \\
  &= \sum_{\substack{(i,\alpha)}} \sum_{j=i+1}^{n-1} \dim V^\alpha_{i;j}\dim W_{j}
  + \sum_{(i,\alpha)} \dim V^\alpha_{i;i} (\dim W_i - \alpha)
\\
  & = \sum_{\substack{(i,\alpha)}} \sum_{j=i}^{n-1} \dim V^\alpha_{i;j}\dim W_{j}
  - \sum_{(i,\alpha)} \alpha \dim V^\alpha_{i;i}.
  \end{split}
\end{equation*}
\end{NB}

Let us fix $\bw$ and consider Poincar\'e polynomials of $\sO(\bv,\bw)$
for various $\bv$. To consider their generating function we introduce
the following notation: Let $e^{\bv}$ be an element of
the group ring of $\Z[I]$ given by
\begin{equation*}
  e^{\bv} \defeq e^{v_1\alpha_1+\cdots + v_{n-1}\alpha_{n-1}},
\end{equation*}
where $\bv = (v_1,\dots, v_{n-1})$. Then we have the following
infinite product expansion.
\begin{Theorem}\label{thm:Poincare}
  \begin{equation*}
    \sum_\bv P_t(\sO(\bv,\bw)) e^{\bv} 
    = \prod_{i\in I} \prod_{\alpha=1}^{\dim W_i} \prod_{j=i}^{n-1}
    \frac1{1 - t^{\dim W_i+ \cdots + \dim W_j - \alpha} e^{\alpha_i+\cdots + \alpha_j}}.
  \end{equation*}
\end{Theorem}

When all $\dim W_i = 1$, the right hand side is essentially equal to
Lusztig's $q$-analog of Kostant's partition function, as first noticed
by Kuznetsov \cite{Kuz}.

\begin{Corollary}
  $\Q$ is connected, provided it is nonempty.
\end{Corollary}

\begin{proof}
  We study a fixed point $\vec{Y}$ with $\dim U_{\vec{Y}} = 0$. From
  the second equality in \eqref{eq:attracting} we have
  \begin{equation*}
    \dim U_{\vec{Y}} = \sum_{\substack{(i,\alpha)}} \sum_{j=i+1}^{n-1} \dim V^\alpha_{i;j}\dim W_{j}
  + \sum_{(i,\alpha)} \dim V^\alpha_{i;i} (\dim W_i - \alpha).
  \end{equation*}
  Suppose this vanishes. Since each term is nonnegative, both must
  vanish.
  The second term vanishes if and only if $Y^\alpha_i = \emptyset$
  unless $\alpha = \dim W_i$, as $V_{i;i}^\alpha = 0$ implies
  $V_i^\alpha = 0$.

  Next consider the first term. Let $i^+$ be the minimum among numbers
  $j$ such that $W_{j} \neq 0$ and $j > i$. If there is no such
  numbers, we set $i^+ = n$. Then the first term vanishes if and only
  if $\dim V^\alpha_{i;j} = 0$ for $j\ge i_+$. Together with the above
  consideration, we find that only $V_i^{\dim W_i}$ contributes to
  $V_j$ for $i\le j < i^+$. Thus $Y_i^{\dim W_i}$ is determined from
  $\bv$. This means the assertion.
\end{proof}

\subsection{Smallness}

We continue the setting of the previous subsection, and assume further
the following condition:
\begin{equation}\label{eq:ass}
  \dim W_1 \le \dim W_2 \le \cdots \le \dim W_n.
\end{equation}
Then \eqref{eq:attracting} implies
\begin{equation}\label{eq:est}
  \dim \sO = \max_{\vec{Y}} \dim U_{\vec{Y}} 
  \le \frac12 \sum_i \dim V_i(\dim W_i + \dim W_{i+1}) - 1
  = \frac12 \dim \Q - 1,
\end{equation}
unless $V = 0$.

\begin{Theorem}
  Under \eqref{eq:ass} the map $\pi\colon \Q \to \Q_0$ is small.
\end{Theorem}

This was proved in \cite{Kuz} when $\dim W_1 = \dots = \dim W_n = 1$,
and his proof works in general cases.
We give a proof for completeness.

\begin{proof}
  Since we need to treat various stratum, let us specify the dimension
  vector $\bv$, $\bw$. We consider the stratification \eqref{eq:stratum}. 

  Recall that the fiber of a point in the stratum
\(
   \Q_0^\reg(\bv-\bv',\bw) \times S^{\bv'}_\Gamma\C
\)
is isomorphic to
\(
   \sO(\bv'_1,\bw)\times\cdots\times \sO(\bv'_m,\bw).
\)
Then \eqref{eq:est} applied to $\Q(\bv'_\alpha,\bw)$ gives
\begin{equation}\label{eq:fiber}
  \dim \pi^{-1}(x) \le
  \frac12 \left( \dim \Q(\bv'_1,\bw) + \dots + \dim \Q(\bv'_m,\bw)
     \right) - m
  = \frac12 \dim \Q(\bv',\bw) - m
\end{equation}
if $x$ is a point in the stratum. We thus have
\begin{equation*}
   \dim \pi^{-1}(\Q_0^\reg(\bv-\bv',\bw) \times S^{\bv'}_\Gamma\C)
   \le \dim \Q_0^\reg(\bv-\bv',\bw) +
   \frac12 \dim \Q(\bv',\bw).
\end{equation*}
This is strictly smaller than $\dim \Q(\bv,\bw)$ unless $\bv' = 0$. It
means that $\Q_0^\reg(\bv,\bw)$ cannot be empty. Hence the dimension
of $\Q_0(\bv,\bw)$ is equal to $\dim \Q(\bv,\bw)$, and the codimension
of the stratum
\(
   \Q_0^\reg(\bv-\bv',\bw) \times S^{\bv'}_\Gamma\C
\)
is equal to
\begin{equation*}
  \dim \Q(\bv',\bw) - m.
\end{equation*}
Therefore the dimension of the fiber $\pi^{-1}(x)$ is smaller than the
half of codimension of the stratum unless $m = 0$, i.e., $\bv' = 0$.
\end{proof}

\subsection{Freeness}

We will consider equivariant homology groups of $\Q(\bv,\bw)$ and
$\sO(\bv,\bw)$ with respect to $\tG_\bw$-actions later.
We state the following:

\begin{Theorem}\label{thm:free}
  $H^{\tG_\bw}_*(\Q(\bv,\bw))$ and $H^{\tG_\bw}_*(\sO(\bv,\bw))$ are
  free of finite rank as $H^*_{\tG_\bw}(\mathrm{pt}) =
  S(\tG_\bw)$-modules.
\end{Theorem}

This is a general property of a variety with torus action such that
Bialynicki-Birula decomposition gives an $\alpha$-partition (see e.g.,
\cite[\S7.1]{Na-qaff}).

\section{Hecke correspondences}\label{sec:Hecke}

We introduce Hecke correspondences in products of handsaw quiver
varieties. They will give generators $E_i^{(r)}$, $F_i^{(r)}$ of the
finite $W$-algebra in \secref{sec:conv}. These are given essentially
in \cite{BFFR}, but we reformulate them because 
\begin{enumerate}
\item their analogy to
Hecke correspondences \cite[\S4]{Na-alg} for ordinary quiver varieties
will become clear;

\item a modification from \cite{BFFR} seems more natural.
\end{enumerate}

\subsection{Definition}

Fix $i\in I$.
Take dimension vectors $\bw$, $\bv^1$, $\bv^2$ such that $\bv^2 =
\bv^1 + \alpha_i$, and consider the product
$\Q(\bv^1,\bw)\times\Q(\bv^2,\bw)$. We denote by $V_i^1$ (resp.\
$V_i^2$) the vector bundle $V_i\boxtimes \shfO_{\Q(\bv^2,\bw)}$
(resp.\ $\shfO_{\Q(\bv^1,\bw)}\boxtimes V_i$). We regard $B_1^p$,
$B_2^p$, $a^p$, $b^p$ ($p=1,2$) as vector bundle homomorphisms between
$V_i^p$, $W_i^p$'s.

We introduce the following variant of the complex \eqref{eq:cpx}:
\begin{equation}\label{eq:cpxHecke}
  \bL(V^2,V^1) 
  \overset{\sigma}{\longrightarrow}
     \begin{matrix} \bE(V^2, V^1) \oplus \bL(V^2,V^1) \\ 
                        \oplus \\
                       \bL(W, V^1) \oplus \bE(V^2, W)
                       \end{matrix}
   \overset{\tau}{\longrightarrow} \bE(V^2,V^1),
\end{equation}
where $\sigma$ and $\tau$ are defined by
\begin{equation*}
   \sigma(\xi) = \begin{pmatrix} \xi B^2_1 - B^1_1 \xi \\
                            \xi B^2_2 - B^1_2 \xi \\
                            \xi a^2 \\
                            - b^1\xi \end{pmatrix}, \quad
   \tau \begin{pmatrix} C_1 \\ C_2 \\ A \\ B \end{pmatrix}
        = B^1_1C_2 - C_2 B^2_1 + C_1B^2_2 - B^1_2 C_1 + a^1B + Ab^2.
\end{equation*}
This is a complex thanks to $\mu = 0$, and by the same argument as in
\cite[5.2]{Na-alg}, $\sigma$ is injective and $\tau$ is surjective,
both pointwise on $\Q(\bv^1,\bw)\times\Q(\bv^2,\bw)$. Therefore
$\Ker\tau/\Ima\sigma$ is a vector bundle of rank
\begin{equation*}
  \dim \bL(W,V^1) + \dim \bE(V^2,W) = 
   \sum_j \left(\dim V^1_j \dim W_j + \dim V^2_j \dim W_{j+1}\right).
\end{equation*}
We define its section $s$ by
\begin{equation*}
  s = (0\oplus 0 \oplus a^1 \oplus - b^2)\mod \Ima\sigma.
\end{equation*}

Let $Z(s)$ denote the zero set of $s$. Then there exists
$\xi\in\bL(V^2,V^1)$ such that
\begin{equation*}
  \xi B_1^2 = B_1^1 \xi, \quad
  \xi B_2^2 = B_2^1 \xi, \quad
  \xi a^2 = a^1, \quad b^2 = b^1\xi.
\end{equation*}
The stability condition implies that $\xi$ is surjective. Then
$(B^1_1,B_2^1,a^1,b^1)$ is isomorphic to data on the quotient
$V^2/\Ker\xi$ induced from $(B^2_1, B^2_2, a^2, b^2)$.

Since $\bv^2 = \bv^1 + \alpha_i$, $\Ker\xi$ must be $1$-dimensional on
$i$ and $0$ otherwise. It is preserved by $B_2^2$ by one of the above
equations. Thus $B_2^2|\Ker\xi$ is a scalar. The decomposition $\Q =
\Q^0\times \C$ induces a decomposition $Z(s) =
\Pa_i(\bv^2,\bw)\times\C$, where $\Pa_i(\bv^2,\bw)$ is defined by
$B_2^2|\Ker\xi = 0$.

Again by the same argument as in \cite[5.7]{Na-alg}, the
differential $\nabla s$ is surjective on $Z(s)$. Therefore $Z(s)$ and
hence $\Pa_i(\bv^2,\bw)$ is a nonsingular subvariety.
\begin{NB}
Its codimension in $\Q(\bv^1,\bw)\times \Q(\bv^2,\bw)$ is
\begin{equation*}
  \sum_j \left(\dim V^1_j \dim W_j + \dim V^2_j \dim W_{j+1}\right) + 1.
\end{equation*}
We have
\begin{multline*}
  \sum_j \left(\dim V^1_j \dim W_j + \dim V^2_j \dim W_{j+1}\right) + 1
  - \frac12 \left( \dim \Q(\bv^1,\bw) + \dim \Q(\bv^2,\bw) \right)
\\
  =  \frac12 \left( \dim W_{i+1} - \dim W_i\right) + 1.
\end{multline*}
This explain the reason why $\dim W_{i+1} - \dim W_i$ appears, but
instead of dividing $\dim W_{i+1} - \dim W_i$ by two, we shift by
$s_{i+1,i}$ and $s_{i,i+1}$ (recall $s_{i+1,i}+s_{i,i+1} = \dim
W_{i+1} - \dim W_i$).
\end{NB}

Note that the construction is $\tG_\bw$-equivariant. In particular,
$\Pa_i(\bv^2,\bw)$ has an induced $\tG_\bw$-equivariant structure.

Remark only $Z(s)$ was considered in \cite{BFFR}.

\subsection{Line bundle}

From the construction in the previous section, we have a natural line
bundle $\mathcal L_i$ over $\Pa_i(\bv^2,\bw)$ given by $\Ker\xi$. This
is a $\tG_\bw$-equivariant line bundle. 
Let us denote the pull-back of the character $q$ under the projection
$\tG_\bw \to \C^*$ also by $q$ for brevity. We twist $\mathcal L_i$ by
as
\begin{equation*}
  \mathcal L'_i \defeq q^{1-i}\mathcal L_i.
\end{equation*}

\begin{NB}
\subsection{tautological complex}

\begin{equation*}
  V_{i-1}
  \overset{\sigma}{\longrightarrow}
     V_{i} \oplus V_{i-1} \oplus W_{i}
   \overset{\tau}{\longrightarrow} V_{i},
\end{equation*}
where $\sigma$ and $\tau$ are defined by
\begin{equation*}
   \sigma = \begin{pmatrix} B_1 \\
                            B_2 \\
                            b \end{pmatrix}, \quad
   \tau = \begin{pmatrix} -B_2 & B_1 & a \end{pmatrix}.
\end{equation*}
\end{NB}

\section{Convolution algebra}\label{sec:conv}

We fix a $\tI$-graded vector space $W$ with dimension vector $\bw$
throughout this section. We also choose a pyramid $\pi$ so that its
row lengths $p_1$, \dots, $p_n$ are given by $\dim W_1$, \dots, $\dim
W_n$. Hence we assume the condition \eqref{eq:ass}.

In \cite{BFFR} a $W^\hbar(\pi)$-module structure was constructed on
\begin{equation*}
  \bigoplus_\bv H^{\tG_\bw}_*(\Q(\bv,\bw))
  \otimes_{S(\tG_\bw)} \mathcal S(\tG_\bw),
\end{equation*}
the direct sum of the localized equivariant cohomology group of Laumon
spaces. Here the variable $\hbar$ for the Rees algebra is identified with
the generator $\hbar$ of $\C[\hbar] = H^*_{\C^*}(\mathrm{pt})$.
More precisely the Rees algebra with respect to the shifted Kazhdan
filtration.

In this section we slightly modify this result working on
non-localized equivariant homology group and using the framework in
\cite[\S8]{Na-alg}, \cite[\S9]{Na-qaff}. The author believes that this
version is more natural, as
\begin{enumerate}
\item it fits with the theory of a convolution algebra,
\item it gives the Rees algebra for the Kazhdan filtration, instead of
  the shifted one.
\end{enumerate}

\subsection{Closed embedding}

Suppose that $V^1$ is an $I$-graded vector space and $V^2$ is its
$I$-graded subspace. We consider $\bM(\bv^1,\bw)$, $\bM(\bv^2,\bw)$ as
in \subsecref{subsec:def}, where $\bv^1$, $\bv^2$, $\bw$ are dimension
vectors of $V^1$, $V^2$, $W$ as before. We embed $\bM(\bv^2,\bw)$ into
$\bM(\bv^1,\bw)$ by choosing a complementary subspace of $V^2$ in
$V^1$. It induces a closed embedding
\[
  \Q_0(\bv^2,\bw) \to \Q_0(\bv^1,\bw),
\]
as in \cite[Lemma~2.5.3]{Na-qaff}. It is independent of the choice of
the complementary subspace.

\subsection{A Steinberg type variety}

Suppose that two dimension vectors $\bv^1$, $\bv^2$ are given. We have
closed embeddings $\Q_0(\bv^1, \bw)$, $\Q_0(\bv^2,\bw)\hookrightarrow
\Q_0(\bv^1+\bv^2,\bw)$. Then we introduce the following variety, as
analog of Steinberg variety defined as in \cite[\S7]{Na-alg}:
\begin{Definition}
Let us consider $\pi$ as a morphism from $\Q(\bv^a,\bw)$ to
$\Q_0(\bv^1+\bv^2,\bw)$ and introduce the fiber product
\begin{equation*}
    Z(\bv^1,\bv^2,\bw)
    \defeq \{ (x^1,x^2)\in \Q(\bv^1,\bw)\times\Q(\bv^2,\bw) \mid
    \pi(x^1) = \pi(x^2)\}.
\end{equation*}
This is called a {\it Steinberg type variety}.
\end{Definition}

Take $([B^1_1,B^1_2,a^1,b^1], [B^2_1,B^2_2,a^2,b^2])\in
\Pa_i(\bv^2,\bw)$. Then $\pi([B^1_1,B^1_2,a^1,b^1])$,
$\pi([B^2_1,B^2_2,a^2,b^2])$ are the same point. However if we only
suppose $([B^1_1,B^1_2,a^1,b^1], [B^2_1,B^2_2,a^2,b^2])\in Z(s)$, then
traces of powers of $B_2^2$ may differ from those of $B_2^1$ as
$B_2^2|\Ker\xi$ possibly is nonzero. This gives a reason why
$\Pa_i(\bv^2,\bw)$ is more natural.

\subsection{Convolution algebra}

We now consider three dimension vectors $\bv^1$, $\bv^2$, $\bv^3$ and
the triple product
$\Q(\bv^1,\bw)\times\Q(\bv^2,\bw)\times\Q(\bv^3,\bw)$. The projection
to the product $\Q(\bv^a,\bw)\times\Q(\bv^b,\bw)$ of the
$a^{\mathrm{th}}$ and $b^{\mathrm{th}}$ factors is denoted by $p_{ab}$.

The map 
\begin{equation*}
  p_{13}\colon p_{12}^{-1}Z(\bv^1,\bv^2,\bw)\cap 
               p_{23}^{-1}Z(\bv^2,\bv^3,\bw)\to
               \Q(\bv^1,\bw)\times \Q(\bv^3,\bw)
\end{equation*}
is proper, and its image is contained in $Z(\bv^1,\bv^3,\bw)$. Hence
we can define the convolution product
\begin{equation*}
  H^{\tG_\bw}_*(Z(\bv^1,\bv^2,\bw))\otimes
  H^{\tG_\bw}_*(Z(\bv^2,\bv^3,\bw))
  \to
  H^{\tG_\bw}_*(Z(\bv^1,\bv^3,\bw)).
\end{equation*}

Taking direct sum over all $\bv^1$, $\bv^2$, we have an algebra
structure on $\bigoplus H^{\tG_\bw}_*(Z(\bv^1,\bv^2,\bw))$. However,
this does not contain the unit, which is the direct {\it product\/}
$\prod_{\bv} [\Delta\Q(\bv,\bw)]$ of the classes of diagonals
$\Delta\Q(\bv,\bw) \subset Z(\bv,\bv,\bw)$. It is practically no harm
at all, as we can consider $[\Delta\Q(\bv,\bw)]$ as a collection of
idempotents. But it does not fit with the usual convention for the
$W$-algebra, which has unit.
Therefore as in \cite[\S9.1]{Na-qaff}, we introduce a slightly larger
space $\prod' H^{\tG_\bw}_*(Z(\bv^1,\bv^2,\bw))$, consisting of
elements $(F_{\bv^1,\bv^2})$ of $\prod
H^{\tG_\bw}_*(Z(\bv^1,\bv^2,\bw))$ such that $F_{\bv^1,\bv^2} = 0$ for
all but finitely many choices of $\bv^2$ for fixed $\bv^1$ and the
same is true for $1$ and $2$ exchanged. Then $\prod'
H^{\tG_\bw}_*(Z(\bv^1,\bv^2,\bw))$ is an algebra with unit.

Let $Z(\bw)$ denote the disjoint union $\bigsqcup Z(\bv^1,\bv^2,\bw)$
and define $H^{\tG_\bw}_*(Z(\bw))$ as the above space
$\prod' H^{\tG_\bw}_*(Z(\bv^1,\bv^2,\bw))$.

We also denote by $\Q(\bw)$ (resp.\ $\sO(\bw)$) the disjoint union
$\bigsqcup \Q(\bv,\bw)$ (resp.\ $\sO(\bv,\bw)$). For these we set
$H^{\tG_\bw}_*(\Q(\bw)) = \bigoplus H^{\tG_\bw}_*(\Q(\bv,\bw))$ and
$H^{\tG_\bw}_*(\sO(\bw)) = \bigoplus
H^{\tG_\bw}_*(\sO(\bv,\bw))$. These are modules of
$H^{\tG_\bw}_*(Z(\bw))$.

\subsection{Image of generators}\label{subsec:image}

The rest of this section is devoted to a construction of a
homomorphism $\rho$ from the finite $W$-algebra $W^\hbar(\pi)$ to the
convolution algebra $H^{\tG_\bw}_*(Z(\bw))$. We first define the
assignment on generators $E_i^{(r)}$, $F_i^{(r)}$, $D_i^{(r)}$ in this
subsection.

We set
\begin{equation*}
  \begin{split}
    \rho(E_i^{(r)}) &\defeq \prod_{\bv} c_1(\mathcal L'_i)^{r-s_{i,i+1}-1}\cap
    [{}^{\omega}\!\Pa_i(\bv,\bw)],
    \\
    \rho(F_i^{(r)}) &\defeq - \prod_{\bv} c_1(\mathcal L'_i)^{r-s_{i+1,i}-1}\cap
    [\Pa_i(\bv,\bw)].
  \end{split}
\end{equation*}
Note that these are different from ones in \cite[\S4]{BFFR}, where
$Z(s)$ in \secref{sec:Hecke} was used instead of
$\Pa_i(\bv,\bw)$. More precisely, the above are multiplied by $\hbar$
from the original ones, which are {\it not\/} contained in
$H^{\tG_\bw}_*(Z(\bw))$.

To define the image of $D_i^{(r)}$, we recall generating functions
$D_i(u)$, $C_i(u)$ and $Z_N(u)$ in \eqref{eq:genD}, and 
${\mathsf A}_i(u)$ in \eqref{eq:sfA}.

For $i\in\tI$ let $\underline{\mathcal E}_i$ be the universal bundle
over $\Q\times\proj^1$ given by considering $\Q$ as the moduli space
of flags 
\(
  0 = E_0 \subset E_1 \subset \cdots \subset E_{n-1} 
  \subset E_n = W\otimes \shfO_{\proj^1}
\)
of locally free sheaves on $\proj^1$ as in \secref{sec:Laumon}.
(This is denoted by $\underline{\mathcal W}_i$ in \cite{BFFR}.)
Recall that it is given by $\Ker\alpha_i$ in the quiver description,
where $\alpha_i$ was defined in \eqref{eq:alpha}.

We modify $\alpha_i$ to a homomorphism on $\Q$ as
\begin{equation*}
  \begin{gathered}
  \bar\alpha_i \colon W_1\oplus\cdots\oplus W_i\oplus q^{-1}V_i
  \to V_i
\\
  \bar\alpha_i \defeq (B_1^{i-1}a, \dots, a, - B_2),\qquad   
  \end{gathered}
\end{equation*}
and define $\mathcal E_i$ as $\Ker\bar\alpha_i$. Here $q^{-1}$ is
introduced so that $\bar\alpha_i$ is $\tG_\bw$-equivariant. Let
\begin{equation*}
  c_{1/u}(\mathcal E_i) = \sum_{r=0}^\infty u^{-r} c_r(\mathcal E_i)
\end{equation*}
be the generating function of Chern classes of $\mathcal E_i$.
We define
\begin{equation*}
  \rho(\mathsf A_i(u)) \defeq u^{\rank \mathcal E_i} c_{1/u}(\mathcal E_i),
\end{equation*}
where $\mathsf A_i(u)$ is the Rees algebra version of the original
$\mathsf A_i(u)$, which means that we replace various $u-k$ by $u -
k\hbar$.

This is again slightly different from one in \cite[\S4]{BFFR}. Our
definition gives
\(
   \prod_{j = 1}^i \prod_{a=1}^{p_j} ( u - \mathsf p_{ij}^{(a)})
\)
instead of
\(
   \prod_{j=1}^i \prod_{a=1}^{p_j} ( u - \hbar^{-1}\mathsf p_{ij}^{(a)})
\)
as the eigenvalue on an element corresponding to the $\hT_{\bw}$-fixed
point, where $\mathsf p_{ij}^{(a)}$ is defined in \cite[(3.15)]{BFFR}.

Consider the last operator $Z_N(u) = A_n(u)$ in the above
definition. Since $\mathcal E_n = W_1\oplus \cdots\oplus W_n$ is a
trivial vector bundle of rank $N$ over $\Q$ (with a nontrivial
$\tG_\bw$-equivariant structure), we have $\rho(Z_N(u))\in
H^*_{\tG_\bw}(\mathrm{pt})[u] = S(\tG_\bw)[u]$. (More precisely, it is
a polynomial of degree $N$ whose leading term is $u^N$.)
This will be compatible with the fact that the coefficients of
$Z_N(u)$ generate the center $Z(W(\pi))$ of $W(\pi)$.

\subsection{A homomorphism from the $W$-algebra to the convolution
  algebra}

The convolution algebra $H^{\tG_\bw}_*(Z(\bw))$ is an algebra over
$H_{\tG_\bw}^*(\mathrm{pt}) = S(\tG_\bw)$. Let
$H^{\tG_\bw}_*(Z(\bw))/\mathrm{tor}$ be the image of
$H^{\tG_\bw}_*(Z(\bw))$ in $H^{\tG_\bw}_*(Z(\bw))\otimes_{S(\tG_\bw)}
\mathcal S(\tG_\bw)$.

\begin{Theorem}\label{thm:BFFR}
The assignment $\rho$ induces an algebra homomorphism
\begin{equation*}
  \rho\colon W^{\hbar}(\pi)\otimes_{\C[\hbar]}\C[\hbar,\hbar^{-1}]
  \to \left(H^{\tG_\bw}_*(Z(\bw))/\mathrm{tor}\right)
  \otimes_{\C[\hbar]}\C[\hbar,\hbar^{-1}].
\end{equation*}
Moreover the center $Z(W^{\hbar}(\pi))$ of $W^\hbar(\pi)$ is mapped
isomorphically to $S(\tG_\bw) = H^*_{\tG_\bw}(\mathrm{pt})$ under
$\rho$.
\end{Theorem}

The last assertion has been explained just before.
The first statement follows from the proof of the main result in
\cite{BFFR}. Let us explain how one should modify it.

First note that the pyramid is {\it left-justified\/} (or equivalently
$s_{i+1,i} = 0$ for all $i$) in \cite{BFFR}, while ours is
arbitrary. This is not essential as there is an explicit algebra
isomorphism $W(\pi) \cong W(\dot\pi)$ for two pyramids with the same
row length \cite[\S3.5]{BruKle}. Thus we may assume $\pi$ is
left-justified.

Next we check that the image of generators $E_i^{(s)}$, $F_i^{(s)}$,
$D_i^{(s)}$ satisfy the defining relation in
\subsecref{subsec:rel}.
In $H^{\tG_\bw}_*(Z(\bw))$, we replace the group $\tG_{\bw}$ by its
maximal torus $\hT_{\bw}$ and take the tensor product with $\mathcal
S(\hT_{\bw})$.
Then the target of $\rho$ is replaced by
\(
  H^{\hT_{\bw}}_*(Z(\bw)) \otimes_{S(\hT_{\bw})} \mathcal S(\hT_{\bw}).
\)
Note that this process kills the torsion part of
$H^{\tG_{\bw}}_*(Z(\bw))$, but is injective on
$H^{\tG_{\bw}}_*(Z(\bw))/\mathrm{tor}$.
By the localization theorem in the equivariant homology it is
isomorphic to
\(
  H_*(Z(\bw)^{\hT_{\bw}}) \otimes_{\C} \mathcal S(\hT_{\bw}).
\)
From the analysis in \secref{sec:fixed}, we have
\(
    Z(\bw)^{\hT_{\bw}} = \Q(\bw)^{\hT_{\bw}}\times \Q(\bw)^{\hT_{\bw}}.
\)
Now it is enough to check the relation in the faithful representation
$H_*(\Q(\bw)^{\hT_{\bw}})\otimes_\C \mathcal S(\hT_{\bw})$. This is
exactly the calculation done in \cite{BFFR} except that we need to
check that our modification of generators gives the Rees algebra with
respect to the Kazhdan's filtration, not the shifted one.

Finally higher root vectors $E_{i,j}^{(r)}$, $F_{i,j}^{(r)}$ are
generated by generators if we invert $\hbar$ (see
\eqref{eq:root}). Therefore we got the assertion.

The author conjectures that $H_*^{\tG_{\bw}}(Z(\bw))$ is torsion free
and $\rho$ gives an isomorphism
\begin{equation*}
  W^\hbar(\pi)\to H_*^{\tG_{\bw}}(Z(\bw)).
\end{equation*}

As a first step towards this conjecture, it should be possible to show
that $\rho$ gives a homomorphism $W^\hbar(\pi)\to
H^{\tG_\bw}_*(Z(\bw))/\mathrm{tor}$, i.e., images of higher root
vectors are in $H^{\tG_\bw}_*(Z(\bw))/\mathrm{tor}$. This is related
to the commutativity of products on $H^{\tG_\bw}_*(Z(\bw))$ at $\hbar
= 0$, as $W^\hbar(\pi)$ is specialized to the coordinate ring of the
Slodowy slice. The author plans to come back to this problem in
future.

\section{Standard modules = Verma modules}\label{sec:std}

In this section, we identify Verma modules with natural modules of the
convolution algebra $H^{\tG_\bw}_*(Z(\bw))/\mathrm{tor}$, called {\it
  standard modules}. Remark that the same terminology was used for a
{\it different\/} class of modules in \cite{BruKle}.


\subsection{Standard modules}\label{subsec:std}

Let $M({P})$ be a Verma module of $W(\pi)$ with the Drinfeld
polynomial ${P}$. The center $Z(W(\pi))$ acts on $M({P})$ via
the character $\chi\colon Z(W(\pi))\to \C$ given by
\begin{equation*}
  \chi(Z_N(u)) = P_1(u) P_2(u) \cdots P_n(u).
\end{equation*}
According to the isomorphism $S(\tG_\bw) = H^*_{\tG_\bw}(\mathrm{pt})
\cong Z(W^\hbar(\pi))$, $\chi$ corresponds to a conjugacy class of a
semisimple element $s$ in $\operatorname{Lie} G_\bw$. We consider $s$
as an element in $\operatorname{Lie} \tG_\bw$, putting $\hbar = 1$ in
the $\C^*$-component.
In concrete terms, the characteristic polynomial of the
$\GL(W_i)$-component of $s$ is $P_i(u)$.
Let $\C_s$ be the $S(\tG_\bw)$-module corresponding to $s$.

Let $\Q(\bw)^s$, $Z(\bw)^s$ denote the $s$-fixed point subvariety in
$\Q(\bw)$ and $Z(\bw)$ respectively. 
Here the $s$-fixed point subvariety is the zero locus of the vector
field corresponding to $s$, or equivalently the subvariety fixed by
$\exp(ts)$ for all $t\in\R$.
We use the corresponding notation for other handsaw quiver varieties
also.
The localization theorem of equivariant homology groups (see e.g.,
\cite{Atiyah-Bott-momentmap}) implies the restriction homomorphism
with support in $\Q(\bw)^s\times\Q(\bw)^s$
\begin{equation*}
  H_*^{\tG_\bw}(Z(\bw))\otimes_{S(\tG_\bw)} \C_s
  \to H_*(Z(\bw)^s)
\end{equation*}
is an isomorphism. Moreover if we correct the homomorphism by the
Euler class of the normal bundle of the second factor of
$\Q(\bw)^s\times\Q(\bw)^s$, it becomes an algebra isomorphism,
as in \cite[5.11.10]{CG}.

Similarly we have an isomorphism
\begin{equation*}
  H^{\tG_\bw}_*(\sO(\bw))\otimes_{S(\tG_\bw)} \C_s
  \to H_*(\sO(\bw)^s),
\end{equation*}
which is compatible with the convolution product. Here $\sO(\bw)^s$ is
the $s$-fixed point subvariety in $\sO(\bw)$. 
More generally if we take $x\in \Q_0(\bw)^s$ and its inverse image
$\Q(\bw)^s_x\defeq \pi^{-1}(x)\cap \Q(\bw)^s$, 
\begin{equation*}
  H^{\tG_\bw}_*(\Q(\bw)^s_x)\otimes_{S(\tG_\bw)} \C_s
  \to H_*(\Q(\bw)^s_x)
\end{equation*}
is also compatible with the convolution product.
The previous $\sO(\bw)^s$ is the special case $x = 0$.
Following \cite[\S13]{Na-qaff} we call these {\it standard modules}.

\begin{Remark}
  As we mentioned in the introduction, the finite AGT conjecture is
  formulated for arbitary compact Lie group $K$. However, the only
  space corresponding to $\Q_0(\bw)$ can be defined, and its
  resolution $\Q(\bw)$ cannot be generalized. The convolution algebra
  can be defined as an $\Ext$-algebra of the sum of the
  $\tG_{\bw}$-equivariant intersection cohomology complexes of
  $\Q^\reg_0(\bv,\bw)$ with various $\bv$. Similarly the convolution
  algebra of the fixed point subvariety $\Q^\bullet_0(\bw^\bullet)$
  can be defined.
  However the `correction term' given by the Euler class of the normal
  bundle does not make sense. The author does {\it not\/} know how to
  define a homomorphism from the specialization of the convolution
  algebra to the convolution algebra of the fixed point set.
\end{Remark}

\subsection{Graded handsaw quiver varieties}\label{subsec:gr}

We study the $s$-fixed subvariety $\Q(\bw)^s$ in this subsection.  We
call $\Q(\bw)^s$ the {\it graded handsaw quiver variety\/} in analogy
with graded quiver varieties studied in \cite{Na-qaff,MR2144973}.

Let $x\in\Q(\bw)^s$ and take its representative $(B_1,B_2,a,b)$.
As in the discussion in \subsecref{subsec:Jordan} the point $x$ is
fixed by $s$ if and only if there exists a semisimple element $\xi\in
\operatorname{Lie} G = \bigoplus_{i\in I} \gl(V_i)$ such that
\begin{equation*}
  [B_1,\xi] = 0, \quad
  [B_2,\xi] = B_2, \quad
  -\xi a = -a s, \quad b\xi = (s + 1) b.
\end{equation*}
We decompose $V$, $W$ according to eigenvalues of $\xi$ and $s$
respectively:
\begin{equation*}
  V = \bigoplus V(m), \quad W = \bigoplus W(m).
\end{equation*}
Then the above equations mean that $B_1$, $a$ preserve eigenvalues
while $B_2$, $b$ decrease them by $1$. Combining the condition with
the $I$-grading, we thus have
\begin{gather*}
  B_1(V_i(m)) \subset V_{i+1}(m), \quad B_2(V_i(m))\subset V_i(m-1),
\\
  a (W_i(m)) \subset V_i(m), \quad b(V_i(m))\subset W_{i+1}(m-1).
\end{gather*}

We define {\it graded\/} dimension vectors
\begin{equation*}
  \bv^\bullet = (\dim V_i(m))_{i\in I,m\in\C} \in\Z_{\ge 0}^{I\times\C}, \quad
  \bw^\bullet = (\dim W_i(m))_{i\in \tI,m\in\C} \in\Z_{\ge 0}^{\tI\times\C}.
\end{equation*}
The conjugacy class of $s$ is determined by $\bw^\bullet$, and vice versa.
The value $\bv^\bullet$ is constant on each connected component of $\Q(\bw)^s$.
We denote by $\Q^\bullet(\bv^\bullet,\bw^\bullet)$ the union of components of
$\Q(\bw)^s$ with given $\bv^\bullet$.
We denote by $\bv$, $\bw$ the corresponding vectors
\begin{equation*}
  (\sum_m \dim V_i(m))_{i\in I} \in\Z_{\ge 0}^{I}, \quad
  (\sum_m \dim W_i(m))_{i\in \tI} \in\Z_{\ge 0}^{\tI}
\end{equation*}
respectively.
We change the notation $\Q(\bw)^s$ to $\Q^\bullet(\bw^\bullet)$
accordingly hereafter.

Let us observe that we can group $V(m)$ according to $m\bmod \Z$, and
components of $(B_1,B_2,a,b)$ between different groups vanish. This
means $\Q^\bullet(\bw^\bullet)$ decomposes into a product of smaller
varieties $\Q^\bullet(\bw^{\prime\bullet})$,
$\Q^\bullet(\bw^{\prime\prime\bullet})$, \dots. Thus we may assume all
eigenvalues of $s$ are integers. Then eigenvalues of $\xi$ are also
integers.
This corresponds to the linkage principle \cite[\S6.3]{BruKle}.

\begin{NB}
Let us rename the graded pieces as
\begin{equation*}
  V'_j(m) \defeq V_{j-m}(m), \quad W'_j(m) \defeq W_{j-m}(m).
\end{equation*}
Then the above implies
\begin{gather*}
  B_1(V'_j(m)) \subset V'_{j+1}(m),\quad
  B_2(V'_j(m)) \subset V'_{j-1}(m-1)
\\
  a (W'_j(m)) \subset V'_j(m), \quad
  b (V'_j(m)) \subset W'_j(m-1).
\end{gather*}
Then we further need to renumber as
\end{NB}%
Let us rename the graded pieces as
\begin{equation}\label{eq:trans}
  V'_j(k) \defeq V_{\frac{j-k}2}(\frac{j+k}2), \quad
  W'_j(k) \defeq W_{\frac{j-k+1}2}(\frac{j+k-1}2).
\end{equation}
\begin{NB}
The inverse is given by
\begin{equation*}
  V_i(m) = V'_{i+m}(m-i), \quad W_i(m) = W'_{i+m}(m-i+1).
\end{equation*}
\end{NB}%
Then the above implies
\begin{gather*}
  B_1(V'_j(k)) \subset V'_{j+1}(k-1),\quad
  B_2(V'_j(k)) \subset V'_{j-1}(k-1),
\\
  a (W'_j(k)) \subset V'_j(k-1), \quad
  b (V'_j(k)) \subset W'_j(k-1).
\end{gather*}
Thus $(B_1,B_2,a,b)$ defines a point in the quiver variety of type $A$
for $V_j' = \bigoplus_k V_j'(k)$, $W'_j = \bigoplus_k W'_j(k)$ as in
\remref{rem:typeA}.
Moreover if we also regard the grading given by $k$, the above
condition is nothing but that of the {\it graded quiver variety\/}
introduced in \cite[\S4.1]{Na-qaff} and \cite[\S4]{MR2144973}.
We denote it by ${\mathfrak
  M}^\bullet(\bv^{\prime\bullet},\bw^{\prime\bullet})$, where the
dimension vectors are transformed by the rule \eqref{eq:trans}.

A point in a graded quiver variety gives a point in a graded handsaw
quiver variety conversely.
Thus we get the following trivial, but main result in this paper:
\begin{Theorem}\label{thm:gqv}
  The graded handsaw quiver variety
  $\Q^\bullet(\bv^\bullet,\bw^\bullet)$ is isomorphic to the graded
  quiver variety ${\mathfrak
    M}^\bullet(\bv^{\prime\bullet},\bw^{\prime\bullet})$.
\end{Theorem}

This result allows us to apply various results for graded quiver
varieties to handsaw quiver varieties
$\Q^\bullet(\bv^\bullet,\bw^\bullet)$.
For example, $\Q^\bullet(\bv^\bullet,\bw^\bullet)$ is connected,
provided it is nonempty (\cite[Th.~5.5.6]{Na-qaff}).
It has no odd cohomology groups \cite[Prop.~7.3.4]{Na-qaff}. We
continue study in subsequent subsections.

\begin{Remark}
  The graded component $V'_j(k)$ was denoted by $V'_j(\ve^k)$ in
  \cite{MR2144973}, where $\ve$ is a nonzero complex number. This is not a
  difference is as $\ve$ is not a root of unity in the current
  setting.

  Note also that $V$ was already an $I\times\C^*$-graded module, and
  we did not put $\bullet$ in the script in \cite{MR2144973}.
\end{Remark}

\subsection{Gelfand-Tsetlin character of a standard module}

We consider $\Q^\bullet(\bv^\bullet,\bw^\bullet)_x \defeq
\Q^\bullet(\bw^\bullet)_x\cap\Q(\bv^\bullet,\bw^\bullet)$, the fiber
of the morphism $\pi$ for graded handsaw quiver varieties.
We have a decomposition
\begin{equation}\label{eq:decomp}
  H_*(\Q^\bullet(\bw^\bullet)_x) 
  = \bigoplus_{\bv^\bullet} H_*(\Q^\bullet(\bv^\bullet,\bw^\bullet)_x).
\end{equation}

\begin{Proposition}
  The above \eqref{eq:decomp} is the $\ell$-weight space decomposition of
  the standard module $H_*(\Q^\bullet(\bw^\bullet)_x)$, considered as
  a $W(\pi)$-module.
\end{Proposition}

The corresponding statement for graded quiver varieties were shown in
\cite[Prop.~13.4.5]{Na-qaff}. The proof works in our current
setting. Let us sketch it.

\begin{proof}
Recall that we have a $\tG_\bw$-equivariant vector bundle $\mathcal
E_i$ over $\Q(\bw)$.
The operator $\mathsf A_i(u)$ is given by the multiplication of
$u^{\rank \mathcal E_i}c_{1/u}(\mathcal E_i)$.
On $H_*(\Q^\bullet(\bw^\bullet))$ it acts through the restriction
\begin{equation*}
  H^*_{\tG_\bw}(\Q(\bw))\to 
  H^*_{\tG_\bw}(\Q(\bw))\otimes_{H^*_{\tG_\bw}(\mathrm{pt})}\C
  \to H^*(\Q^\bullet(\bw^\bullet)),
\end{equation*}
where the first homomorphism is the specialization at $\chi$, and the
second is the pull-back homomorphism with respect to the inclusion
$\Q^\bullet(\bw^\bullet)\subset \Q(\bw)$. Then the
$H^{>0}(\Q^\bullet(\bw^\bullet))$-part of $u^{\rank \mathcal
  E_i}c_{1/u}(\mathcal E_i)$ acts as a nilpotent operator since the
ordinary cohomology group vanishes in sufficiently high degrees. The
$H^0$-part acts as scalar multiplication on each connected component
of $\Q^\bullet(\bw^\bullet)$, i.e., $\Q^\bullet(\bv^\bullet,\bw^\bullet)$.
And the scalar is given by characteristic polynomials of $\xi$ and
$s$, in other words $\bv^\bullet$ and $\bw^\bullet$ from the
definition of $\mathcal E_i$. Therefore
$H_*(\Q^\bullet(\bv^\bullet,\bw^\bullet)_x)$ is an $\ell$-weight space.
\end{proof}

We denote the corresponding $\ell$-weight by $e^{\bv^\bullet}
e^{\bw^\bullet}$. In concrete terms, $e^{\bw^\bullet}$ is an
$\tI$-tuple of characteristic polynomials $P_1(u)$, \dots, $P_n(u)$ of
the $\GL(W_i)$-components of $s$, and $e^{\bv^\bullet}$ is given by
$(Q_1(u),\dots, Q_n(u))$ where $Q_i(u)$ is the characteristic
polynomial of $\xi$ on the virtual vector space
\begin{equation*}
  q^{-1} V_i - V_i - q^{-1} V_{i-1} + V_{i-1},
\end{equation*}
where $\xi$ acts on $q^{-1} V_i$ as $\xi - \id_{V_i}$. Then
$e^{\bv^\bullet} e^{\bw^\bullet}$ is defined as
$(P_1(u)Q_1(u),\dots, P_n(u)Q_n(u))$.
More concretely it is given by the formula
\begin{equation*}
   P_i(u)Q_i(u) = \prod_m (u - m)^{\dim W_i(m) 
     + \dim V_i(m+1) - \dim V_i(m) - \dim V_{i-1}(m+1) + \dim V_{i-1}(m)}.
\end{equation*}  
In terms of $\bv^{\prime\bullet}$,$\bw^{\prime\bullet}$, it is equal
to the $\ell$-weight defined in \cite[(4.4)]{MR2144973}.

\begin{NB}
Note
  \begin{gather*}
    V_i(m+1) - V_i(m) - V_{i-1}(m+1) + V_{i-1}(m)
\\
    = V'_{i+m+1}(m-i+1) - V'_{i+m}(m-i)
    - V'_{i+m}(m-i+2) + V'_{i+m-1}(m-i+1).
  \end{gather*}
\end{NB}

As a corollary, we have

\begin{Corollary}
  \begin{equation*}
    \ch H_*(\Q^\bullet(\bw^\bullet)_x) = \sum_{\bv^\bullet} 
    \dim H_*(\Q^\bullet(\bv^\bullet,\bw^\bullet)_x) e^{\bv^\bullet}e^{\bw^\bullet}
  \end{equation*}
\end{Corollary}

Let us mention that the above Gelfand-Tsetlin character of a standard
module is given in terms of Young tableaux \cite{MR1988990} since the
graded quiver varieties are of type $A$. (The result is originally due
to Kuniba-Suzuki \cite{KunibaSuzuki}.)

It also follows from the definition of $A_i(u)$ that the summand
\begin{equation*}
  H_*(\Q^\bullet(\bw^\bullet)_x\cap \Q(\bv,\bw))
\end{equation*}
is a generalized weight space decomposition with respect to
$D_i^{(1)}$ considered in \subsecref{subsec:adm}. In particular, we
see that the standard module is admissible, as each
$H_*(\Q^\bullet(\bw^\bullet)_x\cap \Q(\bv,\bw))$ is finite
dimensional, and $\bw$ is a highest weight as we will see in
\lemref{lem:l-high} below.

\subsection{An $\ell$-highest weight vector in a standard module}
\label{subsec:hwv}

Let us continue the study of a standard module.

If we restrict the stratification of $\Q_0(\bw)$ in
\subsecref{subsec:strat} to the fixed point locus
$\Q_0^\bullet(\bw^\bullet)$, the symmetric product parts are all
concentrated at the origin, and hence $x\in\Q_0^\bullet(\bw^\bullet)$
is in $\Q_0^{\bullet\reg}(\bv^\bullet,\bw^\bullet)$ for some
$\bv^\bullet$, extended by $0$. We denote this $\bv^\bullet$ by
$\bv^\bullet_x$.
Let us consider $\Q^\bullet(\bv^\bullet_x,\bw^\bullet)_x$. Since
$\pi\colon \Q(\bv,\bw)\to \Q_0(\bv,\bw)$ is an isomorphism on
$\Q_0^\reg(\bv,\bw)$, it consists of a single point, which we also
denote by $x$.

\begin{Lemma}\label{lem:l-high}
  The fundamental class $[x]$ is an $\ell$-highest weight vector in
  the standard module $H_*(\Q^\bullet(\bw^\bullet)_x)$ with
  $\ell$-weight $e^{\bv^\bullet_x} e^{\bw^\bullet}$.
\end{Lemma}

\begin{proof}
  Since $H_*(\Q^\bullet(\bv^\bullet,\bw^\bullet)_x)$ is
  $1$-dimensional, it is clear that $[x]$ is an eigenvector for any
  $D_i^{(r)}$. Its $\ell$-weight is $e^{\bv^\bullet_x}
  e^{\bw^\bullet}$ as calculated above.

  It is killed by all $E_i^{(r)}$, as
  $\Q^\bullet(\bv^{\prime\bullet},\bw^\bullet)\neq\emptyset$ implies
  each component of $\bv^{\prime\bullet}$ is greater than or equal to
  the corresponding component of $\bv^\bullet$ by
  \cite[Prop.~13.3.1]{Na-qaff}.
\end{proof}

Note that $\mathcal E_{i-1} \subset \mathcal E_{i}$ is a vector
subbundle over $\Q^{\bullet\reg}(\bv^\bullet,\bw^\bullet)$. This
implies that the $\ell$-weight $e^{\bv^\bullet} e^{\bw^\bullet}$ is
$\ell$-dominant if
$\Q^{\bullet\reg}(\bv^\bullet,\bw^\bullet)\neq\emptyset$, since
\begin{equation*}
\frac{u^{\rank \mathcal E_i} c_{1/u}(\mathcal E_i)}
{u^{\rank \mathcal E_{i-1}} c_{1/u}(\mathcal E_{i-1})}
= u^{\rank \mathcal E_i/\mathcal E_{i-1}}
c_{1/u}(\mathcal E_i/\mathcal E_{i-1})
\end{equation*}
is a polynomial in $u$. In particular, the $\ell$-weight in the
previous lemma is $\ell$-dominant.

We deduce the following criterion from \cite[Th.~14.3.2(2)]{Na-qaff},
since $\ell$-dominance in this paper is equivalent to one in
\cite{Na-qaff}.

\begin{Proposition}\label{prop:ldom}
$\Q^{\bullet\reg}(\bv^\bullet,\bw^\bullet)\neq\emptyset$ if and only
if $\Q^\bullet(\bv^\bullet,\bw^\bullet)\neq\emptyset$ and
$e^{\bv^\bullet}e^{\bw^\bullet}$ is $\ell$-dominant.
\end{Proposition}

\subsection{Cyclicity}

We now give a main result in this section.

\begin{Theorem}\label{thm:Verma}
  \textup{(1)} The standard module $H_*(\Q^\bullet(\bw^\bullet)_x)$ is
  an $\ell$-highest weight module with an $\ell$-highest weight vector
  $[x]$.

  \textup{(2)} The standard module $H_*(\Q^\bullet(\bw^\bullet)_x)$ is
  the Verma module with $\ell$-highest weight $e^{\bv^\bullet_x}
  e^{\bw^\bullet}$.
\end{Theorem}

\begin{proof}
  (1) Let us prove the assertion as a consequence of the corresponding
  statement for graded quiver varieties proved in
  \cite[Prop.~13.3.1]{Na-qaff}, after the equivariant $K$-theory is
  replaced by the equivariant homology group as in \cite{Varagnolo}.

  Consider the restriction of $F_i^{(r)}$ to
\[  
   H_*(\Q^\bullet(\bv^{2\bullet},\bw^\bullet))\to
   H_*(\Q^\bullet(\bv^{1\bullet},\bw^\bullet)).
\]
It can be computed in the following way, as explained in
\subsecref{subsec:std}:
\begin{enumerate}
\item replace $\Pa_i(\bv,\bw)$ by the Koszul complex for the section $s$
of $\Ker\tau/\Ima\sigma$ in \eqref{eq:cpxHecke},
\item restrict the complex to
\(
   \Q^\bullet(\bv^{1\bullet},\bw^\bullet)\times
    \Q^\bullet(\bv^{2\bullet},\bw^\bullet),
\)
\item divide by the Euler class (evaluated at $s$) of the normal
  bundle of the second factor $\Q^\bullet(\bv^{2\bullet},\bw^\bullet)$
  in $\Q(\bv^2,\bw)$.
\end{enumerate}
The restriction of the vector bundle $\Ker\tau/\Ima\sigma$ decomposes
into sum of the $s$-fixed part and the other part.
The $s$-fixed part gives a resolution of the intersection
\begin{equation*}
    \Pa_i(\bv,\bw)\cap
    \Q^\bullet(\bv^{1\bullet},\bw^\bullet)\times
    \Q^\bullet(\bv^{2\bullet},\bw^\bullet).
\end{equation*}
The other part is {\it close to\/} the normal bundle of the second
factor: if we replace $V^1$ by $V^2$, it precisely is the normal
bundle, as can be seen from \eqref{eq:cpx}. The difference $V^1/V^2$
is nothing but the definition of the line bundle $\mathcal
L'_i$. Therefore the restriction is given by a linear combination of
correspondences
\begin{equation*}
    c_1(\mathcal L'_i)^{r-s_{i+1,i}}\cap
    [\Pa_i(\bv,\bw)\cap
    \Q^\bullet(\bv^{1\bullet},\bw^\bullet)\times
    \Q^\bullet(\bv^{2\bullet},\bw^\bullet)],
\end{equation*}
multiplied with various Chern classes of $\mathcal L'_i$ and
components of $V^1$, $V^2$. Chern classes of components of $V^1$,
$V^2$ are given by products of various operators $D_j^{(s)}$.
Multiplication of a power of $c_1(\mathcal L'_i)$ is a restriction of
$E_i^{(s)}$ with $s\ge r$ plus correction terms.

The same is true for negative generators of the Yangian. Since the
complex \eqref{eq:cpxHecke} is different, the restriction of
$F_i^{(r)}$ is not the same, but the linear span of products of
$F_i^{(r)}$ and various products of $D_j^{(s)}$ is the same for the
finite $W$-algebra and Yangian. This statement is enough to conclude
$H_*(\Q^\bullet(\bw^\bullet)_x)$ is $\ell$-highest weight.
  
(2) Let us consider the dimension of the weight space
\begin{equation*}
    H_*(\sO^\bullet(\bw^\bullet)\cap \Q(\bv,\bw))
    = H^{\tG_\bw}_*(\sO(\bv,\bw))\otimes_{S(\tG_\bw)} \C_s
\end{equation*}
of the standard module $H_*(\sO^\bullet(\bw^\bullet))$.
By \thmref{thm:free} this is equal to the Euler number of
$\sO(\bv,\bw)$, which is given by the formula in
\thmref{thm:Poincare}. This coincides with one given by
\cite[Th.~6.1]{BruKle}. Therefore the assertion for the special case $x =
0$ is proved.
  
A general case follows from the special case $x = 0$ from
\cite[Th.~3.3.2 and Rem.~3.3.3]{Na-qaff}.
\end{proof}

\section{Simple modules}\label{sec:simple}

We give a character formula of a simple module in terms of dimensions
of intersection cohomology groups of graded quiver varieties in this
section.

\subsection{IC sheaves}

Combining \propref{prop:ldom} with \thmref{thm:Verma}, we find that
the stratification
\begin{equation*}
  \Q^\bullet(\bw^\bullet) 
  = \bigsqcup_{\bv^\bullet} \Q^{\bullet\reg}(\bv^\bullet,\bw^\bullet),
\end{equation*}
introduced at the beginning of \subsecref{subsec:hwv} is indexed by
the set of vectors such that $e^{\bv^\bullet}e^{\bw^\bullet}$ is an
$\ell$-dominant $\ell$-weight of the Verma module with $\ell$-highest
weight vector $e^{\bw^\bullet}$.
From \thmref{thm:gqv} (or probably from somewhere in \cite{BruKle})
and the corresponding result for graded quiver varieties
\cite[Th.~14.3.2]{Na-qaff}, we know that the set is finite.

We denote by $IC(\Q^{\bullet\reg}(\bv^\bullet,\bw^\bullet))$ the
intersection cohomology complex associated with the constant sheaf on
$\Q^{\bullet\reg}(\bv^\bullet,\bw^\bullet)$.

Take a point $x_{\bv^\bullet,\bw^\bullet}$ from the stratum
$\Q^{\bullet\reg}(\bv^\bullet,\bw^\bullet)$, and denote by
$i_{\bv^\bullet,\bw^\bullet}$ the inclusion $\{ x\} \to
\Q^\bullet(\bw^\bullet)$. The statement of the following theorem holds
for any choice of $x_{\bv^\bullet,\bw^\bullet}$, hence our notation
does not cause any trouble.

Now we can state our main result in this paper.

\begin{Theorem}\label{thm:main}
  The composition multiplicity of $L({Q})$ in $M({P})$ is
  given by
  \begin{equation}\label{eq:compo}
    [M({P}): L({Q})] = 
    \dim H^*(i_{\bv_P^\bullet,\bw^\bullet}^!
    IC(\Q^{\bullet\reg}(\bv_Q^{\bullet},\bw^\bullet))),
  \end{equation}
  where $\bv_P^\bullet$, $\bv_Q^{\bullet}$ are determined by
\(
    e^{\bv_P^\bullet} e^{\bw^\bullet} = {P},
\)
\(
    e^{\bv_Q^{\bullet}} e^{\bw^\bullet} = {Q}.
\)
\end{Theorem}

This is proved exactly as in \cite[Th.~14.3.10]{Na-qaff}. We use
Ginzburg's theory \cite[\S8.6]{CG} of simple modules of a convolution
algebra to analyze $H_*(Z(\bw^\bullet))$. Since the algebra
homomorphism $\rho$ in \thmref{thm:BFFR} may not be an isomorphism, it
is not {\it a priori\/} clear whether a simple
$H_*(Z(\bw^\bullet))$-module remains simple under the pull-back by
$\rho$. But this difficulty can be overcome by the use of the
cyclicity proved in \thmref{thm:Verma}, as in the proof of
\cite[Th.~14.3.2(3)]{Na-qaff}.

\subsection{KL polynomials}\label{sec:kl-polynomials}

As we noticed in \subsecref{subsec:gr} graded handsaw quiver varieties
are isomorphic to graded quiver varieties of type $A$. Using results
known on graded quiver varieties, we identify the right hand side of
\eqref{eq:compo} with the evaluation of a Kazhdan-Luszitg polynomial
at $1$ in this subsection. We give two proofs.

By \cite[Th.~14.3.10]{Na-qaff} the right hand side
of \eqref{eq:compo} is equal to the composition multiplicity for
finite dimensional representations of a quantum affine algebra
$\hat{\mathbf U}_\ve$ specialized at $q = \ve$, which is not a root of
unity:
\begin{equation*}
  [M^{\hat{\mathbf U}_\ve}(P'): L^{\hat{\mathbf U}_\ve}(Q')] 
  = \dim H^*(i_{\bv_P^\bullet,\bw^\bullet}^!
    IC(\Q^{\bullet\reg}(\bv_Q^{\bullet},\bw^\bullet))),
\end{equation*}
where $M^{\hat{\mathbf U}_\ve}(P')$ (resp.\ $L^{\hat{\mathbf
    U}_\ve}(Q')$) is a standard (resp.\ simple) module of
$\hat{\mathbf U}_\ve$. Here the Drinfeld polynomial $P'$ is given by
the formula
\begin{equation*}
  P'_j(u) = \prod_k (1 - \ve^{-k}u)^{\dim W_j'(k) + 
    \dim V'_{j+1}(k) + \dim V'_{j-1}(k) - \dim V'_{j}(k+1) - \dim V'_{j}(k-1)},
\end{equation*}
when $\Q^\bullet(\bv_P^\bullet,\bw^\bullet)$ is identified with
$\mathfrak M^\bullet(\bv^{\prime\bullet},\bw^{\prime\bullet})$ as in
\thmref{thm:gqv}. And the same applies for $Q'$.
\begin{NB}
We change the coefficient of $u$ from $\ve^k$ to $\ve^{-k}$, as
the stability condition is opposite in this paper.
\end{NB}%

If we fix an $\ell$-highest weight $e^{\bw^\bullet}$, we only have a
finitely many $\ell$-dominant $\ell$-weight in the Verma module, and
hence we can work on a quantum affine algebra ${\mathbf
U}_{\ve}(\widehat\algsl_M)$ for some $M$. However, if we only fix the
finite $W$-algebra $W(\pi)$, we cannot have a bound on $M$: if we have
a bound on $j-k$, but not on $j+k$ in \eqref{eq:trans}, both $j$ and
$k$ can arbitrary. Therefore the above $\hat{\mathbf U}_{\ve}$ should
be understood as ${\mathbf U}_{\ve}(\widehat\algsl_\infty)$, defined
as the limit of ${\mathbf U}_{\ve}(\widehat\algsl_M)$ as $M\to
\infty$.

By \cite{Varagnolo} this composition multiplicity is the same as that
for Yangian. And it is given by a Kazhdan-Lusztig polynomial thanks to
Arakawa's result \cite[Th.~16]{Arakawa}.
\begin{NB}
Our convention
\begin{equation*}
   P'_j(u) = \prod_n (1 - \ve^{-k} u)^{\dim W_j'(k)},
\end{equation*}
where $\ve$ is a nonzero complex number, which is not a root of
unity. We change the coefficient of $u$ from $\ve^k$ to $\ve^{-k}$, as
the stability condition is opposite in this paper.
We change it to
\begin{equation*}
   P'_j(u) = \prod_n (u + k)^{\dim W_j'(k)},
\end{equation*}
which is the Drinfeld polynomial for the Yangian.

Arakawa relates Drinfeld polynomials and $\lambda$, $\mu$ by the rule
\begin{equation*}
   P_j'(u) = \prod_{\substack{\alpha=1,\dots,N \\ \lambda_\alpha - \mu_\alpha = j}}
   (u - 2\lambda_\alpha + j + 1).
\end{equation*}
Therefore $k = -2\lambda_\alpha + j + 1$, i.e.,
\begin{equation*}
   \lambda_\alpha = \frac{j-k+1}2, \qquad
   \mu_\alpha = \frac{-j-k+1}2.
\end{equation*}
Comparing with \eqref{eq:trans}, we find
\begin{equation*}
   \lambda_\alpha = i, \qquad \mu_\alpha = - m.
\end{equation*}

This means that the multisegment consists of $[-m+1,i] = [(-j-k+3)/2, (j-k+1)/2]$. This has the length $i+m=j$.

\end{NB}%
Let us state this result in terms of $P$, $Q$. We assume all zeros of
Drinfeld polynomials of $P$ are integers. A general case is reduced to
this case as mentioned in \subsecref{subsec:gr}. We identify integral
weights of $\gl_N$ with elements in $\Z^N$. We define integral weights
$\lambda$, $\mu'$ of $\gl_N$ by
\begin{equation*}
  \begin{split}
  \lambda & \defeq (\underbrace{n,\dots,n}_{\text{$p_n$ times}},
  \underbrace{n-1,\dots,n-1}_{\text{$p_{n-1}$ times}},\dots, 
  \underbrace{1,\dots, 1}_{\text{$p_1$ times}}),
\\
  \mu' &\defeq
  ( -m_n^1,\dots -m_n^{p_n}, -m_{n-1}^1,\dots, -m_{n-1}^{p_{n-1}},\dots,
  -m_1^1,\dots, -m_1^{p_1}),
  \end{split}
\end{equation*}
where $m_i^1$, \dots, $m_i^{p_i}$ are zeros of $P_i(u)$. We take a
dominant weight $\mu$ so that $\mu' = w\mu$ for some $w\in S_N$. Note
that $w$ is well-defined in the double coset $S_\lambda\backslash
S_N/S_\mu$, where $S_\lambda$, $S_\mu$ are stabilizers of $\lambda$,
$\mu$ respectively. We define $\nu'$ from zeros of $Q$ as above. Then
Arakawa showed that $\nu' = x\mu$ for some $x\in S_N$ when the
composition multiplicity is nonzero. We let $x_{LR}$, $w_{LR}$ the
longest length representatives of $x$, $w$ in $S_\lambda\backslash
S_N/S_\mu$.
\begin{Theorem}
We have
  \begin{equation*}
    [M({P}): L({Q})] = P_{w_{LR},x_{LR}}(1),
  \end{equation*}
  where $P_{yz}(q)$ is the Kazhdan-Lusztig polynomial associated with
  $S_N$.
\end{Theorem}

In Arakawa's result there are restrictions on $x$, $w$ coming from the
rank of the underlying Lie algebra $\algsl$. But as we take limit
explained as above, we do not impose such a restriction.

This result gives a new proof of a conjecture by Brundan-Kleshchev
\cite[Conj.~7.17]{BruKle}, which was proved by Losev \cite[Th.~4.1,
Th.~4.3]{Losev}, in a wider context, by a different method.

Arakawa's proof depends on a solution of Kazhdan-Lusztig conjecture for $\algsl_N$. Let us give another more direct proof. (It will give a new proof of Kazhdan-Lusztig conjecture, when $\bw = (1,1,\dots,1)$ and $W(\gl_N,e) = U(\gl_N)$.)

Let $O(\rho)$ be the nilpotent orbit corresponding to an $\algsl_2$-triple $\rho$. Let $S(\rho)$ denote Slodowy slice at
  $\rho\left[
  \begin{smallmatrix}
    0 & 0 \\ 1 & 0
  \end{smallmatrix}\right].$
By \cite[Th.~8.4]{Na-quiver}, a quiver variety $\mathfrak M(\bv',\bw')$ of type $A_\ell$ is an
intersection $\overline{O(\rho)}\cap S(\rho')$ where $\rho$ and
$\rho'$ are given by dimension vectors by the formula
  \begin{equation*}
    \begin{gathered}
    \mathbf u = (u_1,\dots,u_\ell) \defeq \bw' - {\mathbf C} \bv', \quad
    \rho \longleftrightarrow (1^{u_1} 2^{u_2}\cdots \ell^{u_\ell}(\ell+1)^{v'_\ell}),
\\
    \rho' \longleftrightarrow (1^{w'_1} 2^{w'_2}\cdots \ell^{w'_\ell}),
    \end{gathered}
  \end{equation*}
  where $\mathbf C$ is the Cartan matrix of type $A$, and we write partitions corresponding to $\rho$, $\rho'$.
\begin{NB}
  The formula in \cite[Th.~8.4]{Na-quiver} was wrong. $\mu$ must be
  replaced by its transpose.
\end{NB}

The $\C^*$-action on the quiver variety coincides with the standard
$\C^*$-action on the Slodowy slice, defined using the
$\algsl_2$-triple $\rho'$:
\begin{equation*}
    t^2 \operatorname{Ad}\left(\rho\left[\begin{matrix} t^{-1} & 0 \\ 0 & t\end{matrix}\right]
    	\right).
\end{equation*}
(See e.g., \cite[\S7.4]{Slodowy}.) And $\prod_j \GL(W_j')$ is
identified with the centralizer of $\rho'$. Then the graded quiver
variety $\mathfrak
M^\bullet_0(\bv^{\prime\bullet},\bw^{\prime\bullet})$ is identified
with the closure of an orbit $O(\rho^\bullet)$ of the space of
representations of a type $A$ quiver, intersected with a slice
$S(\rho^{\prime\bullet})$ to another orbit.
Those orbits are parametrized by {\it multisegments}, collections of
segments in $\Z$. A segment $[a, b]$ corresponds to an indecomposable
representation
\begin{equation*}
  \cdots \ \ 0 \ \ 0\ \ \C \to \C \to\cdots \to \C\ \  0\ \ 0\cdots, 
\end{equation*}
where $\C$ starts from the vertex $a$, and ends at the vertex $b$. For example, $\rho^{\prime\bullet}$ corresponds to the union of segments
 \begin{equation*}
   [\frac{-j-k+3}2, \frac{j-k+1}2] = [-m+1,i]
\end{equation*}
for each $W'_j(k) =  W_i(m)$. Thus the multisegment is given by $\lambda$, $\mu'$ defined above. The multisegment corresponding to $\rho^\bullet$ is similarly given by $\lambda$, $\nu'$.

Now the dimension of the intersection cohomology group in question
was given by a Kazhdan-Lusztig polynomial by \cite{Zelevinski}. It is known that the Kazhdan-Lusztig polynomial is equal to one given above, thanks to \cite{Henderson}.

\begin{NB}
Instead of \cite[Th.~8.4]{Na-quiver}, we can use \cite{GinzburgVasserot,Vasserot} to describe the
composition multiplicity for $\hat{\mathbf U}_{\ve}$. Then it is given by the intersection cohomology groups of the $\C^*$-fixed point locus of the nilpotent variety, that is the space for a type $A$ quiver as
above. The remaining part is the same.
\end{NB}

Another possible approach is to express the algorithm of \cite{MR2144973}
in terms of Young tableaux by a method in \cite{MR1988990}.
The author does not have ability to determine Weyl group elements
obtained in this approach. So he leaves this problem as an exercise
for the reader.

\subsection*{Acknowledgments}
The author thanks Shintarou Yanagida for his talk explaining
\cite{BFFR}, and
Tomoyuki Arakawa for discussions on $W$-algebras.

\bibliographystyle{myamsplain}
\bibliography{nakajima,mybib,handsaw}

\end{document}